\documentclass[11pt]{book}
\usepackage{amsmath, amscd, amssymb, amsthm}
\usepackage[subnum]{cases}
\usepackage{changepage}
\usepackage{marginnote}
\usepackage{hyperref}
\usepackage{cleveref}
\usepackage[all]{xy}
\usepackage{tabularx}
\usepackage{ltablex}
\usepackage[T1]{fontenc}
\usepackage{setspace}
\usepackage{fancybox}
\usepackage{ragged2e}
\usepackage[margin=1.1in]{geometry}
\hypersetup{colorlinks,linkcolor={black},citecolor={black}}
\usepackage{mathpazo}
\usepackage{enumitem}
\setlist[itemize]{leftmargin=+0.1in}

\newtheorem{theorem}{Theorem}[section]
\newtheorem{lemma}[theorem]{Lemma}

\newtheorem{proposition}[theorem]{Proposition}
\newtheorem{corollary}[theorem]{Corollary}
\newtheorem{conjecture}[theorem]{Conjecture}

\newtheorem{notation}[theorem]{Notation}

\theoremstyle{definition}
\newtheorem{definition}[theorem]{Definition}
\newtheorem{remark}[theorem]{Remark}
\numberwithin{equation}{section}

\newtheorem{assumption}[theorem]{Assumption}
\newtheorem{setting}[theorem]{Setting}

\usepackage{fancyhdr}
\begin{document}

\normalfont

\title{$\infty$-Categorical Perverse $p$-adic Differential Equations over Stacks}
\author{Xin Tong}
\date{}

\maketitle

\newpage

\subsection*{abstract}
\rm We will discuss $\infty$-categorical perverse $p$-adic differential equations over stacks. On one hand, we are going to study some $p$-adic analogous results of the Drinfeld's original lemma about the \'etale fundamental groups in the \'etale setting, in the context of $F$-isocrystals closely after Kedlaya and Kedlaya-Xu. We expect similar things could also be considered for diamonds after Scholze, in the context of Kedlaya-Liu's work namely the derived category of pseudocoherent Frobenius sheaves, which will induce some categorical form of Drinfeld's lemma for diamonds motivated by work of Carter-Kedlaya-Z\'abr\'adi and Pal-Z\'abr\'adi. On the other hand, we are going to establish the $\infty$-categorical theory of arithmetic $D$-modules after Abe and Gaitsgory-Lurie, which will allow one to construct the rigid Gross $G$-motives. And we are expecting to apply the whole machinery to revisit Weil's conjecture parallel to and after Gaitsgory-Lurie.

\newpage

\footnotetext[1]{\textit{Keywords and Phrases}: $p$-adic differential equation, $\infty$-categorical arithmetic $D$-modules, $\infty$-categorical arithmetic $D^\dagger$-modules,  $L$-functions.}

\newpage

\tableofcontents

\newpage

\chapter{Perverse $p$-adic Differential Equations over Quotient Stacks}

Drinfeld's celebrated lemma relates the corresponding product of \'etale fundamental groups of different schemes with the corresponding fundamental group of some single stack coming the quotient of the product space of the involved schemes. This is very important since in \cite{La1} V.Lafforgue relied heavily on the corresponding such lemma to give the corresponding parametrisation after Langlands in the context of function field arithmetic. Following this, Kedlaya conjectured that we should also have the chance to do this in the context of $p$-adic cohomology theory (also see \cite{Ked9} and \cite{KX}), which will generalize the work of \cite{Abe} for linear algebraic groups.

%\newpage

\section{The Drinfeld's Lemma}

\noindent In this section we recall some basics around the corresponding original Drinfeld's Lemma in the context of the \'etale local systems over schemes in characteristic $p>0$.

\begin{assumption}
We will mainly work with smooth schemes which are also assumed to be connected, reduced and separated over $\mathbb{F}_p$.
\end{assumption}

\indent Then in the current situation, suppose we have different schemes $X_1,...,X_I$ over $\mathbb{F}_p$ then we could consider the corresponding product of these schemes in the absolute sense:
\begin{align}
X:=X_1\times_{\mathrm{Spec}\mathbb{F}_p}...\times_{\mathrm{Spec}\mathbb{F}_p}X_I.	
\end{align}
What happens is that one might want to ask if we could have the chance to relate the \'etale local systems over $X$ and the ones over each scheme. However this is not very direct in the sense that we have to consider more construction in order to make such relationship transparent. This is how Drinfeld's Lemma kicks in. Here is how the idea works. First we have the corresponding partial Frobenius (which was observed by Drinfeld) $\varphi_i$ for $i=1,...,I$. We need to consider some sort of multi Frobenius structures in order to guarantee that we could have the desired relationship. The first construction is the following stack for any $i\in I$:
\begin{displaymath}
X/\Phi:=X/\left<\varphi_1,...,\widehat{\varphi}_i,...\varphi_1\right>,
\end{displaymath}
see \cite[Sheaves, Stacks and Shtukas, Definition 4.2.10]{perf} as well for more details. For such stack the corresponding construction is bit complicated.

\begin{remark}
Recall from \cite[Sheaves, Stacks and Shtukas, Definition 4.2.10]{perf} that the corresponding \'etale coverings of the stack $X/\Phi$ could be related to the corresponding \'etale coverings of the scheme $X$, but one has to consider the further equivariant Frobenius action coming from $\varphi_1,...,\varphi_I$ (with the corresponding composite specified as those relative Frobenius $\varphi_{Y/X}$ over $X$ for any covering $Y $), namely for such $Y$ for $X$ we need to assume that $Y$ admits isomorphisms up to all the pullbacks by the partial Frobenius actions. 	
\end{remark}

\indent So we can have the reasonable category $\text{F\'ET}(X/\Phi)$, and we will use the corresponding notation $\text{F\'ET}(X)_{\Phi}$ to denote the corresponding equivalent equivariant category:
\begin{align}
\text{F\'ET}(X)_{\Phi}	\overset{\sim}{\longrightarrow} \text{F\'ET}(X/\Phi).
\end{align}

\indent So we can have the reasonable profinite fundamental group $\pi_1^\mathrm{prof}(X/\Phi,\overline{x})$, and under the corresponding framework of Tannakian cateogories, we will use the notation $\pi_1^\mathrm{prof}(X,\overline{x})_\Phi$ to denote the Tannakian fundamental group of the category $\text{F\'ET}(X)_{\Phi}$ with
\begin{align}
\text{F\'ET}(X)_{\Phi}	\overset{\sim}{\longrightarrow} \text{F\'ET}(X/\Phi).
\end{align}

\begin{proposition} \mbox{\bf{(The Drinfeld's Lemma)}}
We have the isomorphism between the fundamental group of the stack $X/\Phi$ and the product of the ones from each space $X_i$, $i=1,...,I$. Therefore we have the corresponding equivalence between the category of all the $\ell$-adic representations of $\pi_1^\mathrm{prof}(X/\Phi,\overline{x})$	and the category of all the $\ell$-adic representations of the product of fundamental groups of the each separate schemes involved, for any prime number $\ell$ (which could be the same as $p$).
\end{proposition}

\newpage

\section{Convergent Isocrystals}

\indent Now we switch to just the $p$-adic setting in the context of convergent isocrystals.

\begin{setting}
Recall from \cite[Definition 2.1]{Ked1}, we have the corresponding category of all the convergent $F$-isocrystals over some scheme $S$ over $\mathbb{F}_p$. We use the notation $\textbf{FIsoc}(S)$ to denote the category of all the convergent $F$-isocrystals over $S$ (the exact definition could be found in \cite[Definition 2.1]{Ked1}). And at the same time we also have the corresponding category of overconvergent $F$-isocrystals over the same space $S$, which we denote it by $\textbf{FIsoc}^\dagger(S)$. 
\end{setting}

\indent A corresponding convergent $F$-isocrystal over $S$ is vector bundle $M$ over the corresponding rigid analytic generic fiber $S^\mathrm{an}$ over $\mathbb{Q}_p$ of the corresponding lift $\mathcal{S}$ of $S$ to $W(\mathbb{F}_p)$ in the glueing fashion carrying an integrable connection $\nabla$, with the corresponding action coming from the Frobenius lift $\sigma$. In our situation we also have to consider the corresponding vector bundles over some quotient stack $X/\Phi$ ($X$ is the product space in the previous section). Certainly this will mean that we have to consider the corresponding quotient from the corresponding $X$ to this quotient.

\begin{definition}
Over the stack $X/\Phi$ we define a corresponding convergent $F$-isocrystal $M$ to be a convergent $F$-isocrystal $(M,\nabla,\sigma)$ over $X$ with the corresponding action from each partial Frobenius $\varphi_i$ ($i\in I$) compatible with the corresponding partial Frobenius actions on the space $X$. We can consider the corresponding category of all such objects, which will be denoted by $\textbf{FIsoc}(X/\Phi)$. 	
\end{definition}

\begin{proposition}
With the corresponding notations established above. Consider the following three categories. The first is the category of all the $p$-adic representations of $\prod_{I} \pi_1^\mathrm{prof}(X_i,\overline{x}_i)$. The second is the category of all the $p$-adic representations of $\pi_1^\mathrm{prof}(X/\Phi,\overline{x})$. The third is the category of all the convergent $F$-isocrystals over $X/\Phi$ which are unit-root for $\sigma$ when regarded as the corresponding objects over $X$ carrying partial Frobenius actions. Then we have that all the three categories are equivalence. The corresponding functors are the one induced from the usual Drinfeld's lemma on the isomorphism of the groups on the both sides, and the one coming from Katz-Crew's equivalence.	
\end{proposition}

\begin{proof}
This would be just a combination of the usual Drinfeld's lemma as in the previous section, and the corresponding Katz-Crew's equivalence \cite{Ka}, \cite{Cr2} on the unit-root convergent isocrystals and the corresponding $p$-adic \'etale local systems as in \cite[Theorem 2.1]{Cr2}.	
\end{proof}

\begin{corollary}  \label{corollary2.4}
Let $I$ be a set of cardinality 2. And we put $X_2$ to be $\mathrm{Spec} k=\mathrm{Spec}\overline{\mathbb{F}}_p$. Then we have the category of the corresponding unit root convergent $F$-isocrystals over $X_{1,k}/\varphi_k$ and the category of the corresponding unit root convergent $F$-isocrystals over $X_1$ are equivalent.	
\end{corollary}

\indent Beyond the corresponding unit-root situation, we should consider the corresponding non-\'etale context where we need to replace the corresponding profinite \'etale fundamental groups by the corresponding isocrytal Tannakian fundamental groups $\pi^\mathrm{Isoc}_1$ (see \cite[Appendix B]{DK}). Therefore the fundamental conjecture will be the following:

\begin{conjecture}
The category of the corresponding $\overline{\mathbb{Q}}_p$-representations of $\prod_{i=1}^I \pi^\mathrm{Isoc}_1(X_i)$ is equivalent to the corresponding category of all the $\overline{\mathbb{Q}}_p$-representations of $\pi^\mathrm{Isoc}_1(X/\Phi)$, which is further equivalent to the corresponding category of all the $\overline{\mathbb{Q}}_p$-convergent isocrystals over $X/\Phi$.
\end{conjecture}

\newpage
\section{Overconvergent Isocrystals}

\indent Now we switch to just the $p$-adic setting in the context of overconvergent isocrystals.

\begin{setting}
Recall from \cite[Definition 2.7]{Ked1}, we have the corresponding category of all the overconvergent $F$-isocrystals over some scheme $S$ over $\mathbb{F}_p$. We use the notation $\textbf{FIsoc}^\dagger(S)$ to denote the category of all the overconvergent $F$-isocrystals over $S$ (the exact definition could be found in \cite[Definition 2.7]{Ked1}).  
\end{setting}

\begin{definition}
Over the stack $X/\Phi$ we define a corresponding overconvergent $F$-isocrystal $M$ to be an overconvergent $F$-isocrystal $(M,\nabla,\sigma)$ over $X$ with the corresponding action from each partial Frobenius $\varphi_i$ ($i\in I$) compatible with the corresponding partial Frobenius actions on the space $X$. We can consider the corresponding category of all such objects, which will be denoted by $\textbf{FIsoc}^\dagger(X/\Phi)$. 	
\end{definition}

\begin{proposition}
With the corresponding notations established above. Consider the following two categories. The first is the category of all the $p$-adic \'etale local systems over $X/\Phi$ which are potentially unramified with respect to $X$. The second is the category of all the overconvergent $F$-isocrystals over $X/\Phi$ which are unit-root for $\sigma$ when regarded as the corresponding objects over $X$ carrying partial Frobenius actions. Then we have that all the two categories are equivalence. The corresponding functors are the one induced from the usual Drinfeld's lemma on the isomorphism of the groups on the both sides, and the one coming from Crew-Tsuzuki-Kedlaya's equivalence.	
\end{proposition}

\begin{proof}
This would be just a combination of the usual Drinfeld's lemma as in the previous section, and the corresponding Crew-Tsuzuki-Kedlaya's equivalence \cite{Cr1}, \cite{Cr2}, \cite{Tsu}, \cite{Ked4}, \cite{Ked5}, \cite{Ked6}, \cite{Ked7} on the unit-root convergent isocrystals and the corresponding $p$-adic \'etale local systems (for instance see \cite[Theorem 1.3]{Tsu}).	
\end{proof}

\begin{corollary}  \label{corollary3.4}
Let $I$ be a set of cardinality 2. And we put $X_2$ to be $\mathrm{Spec} k=\mathrm{Spec}\overline{\mathbb{F}}_p$. Then we have the category of the corresponding unit root overconvergent $F$-isocrystals over $X_{1,k}/\varphi_k$ and the category of the corresponding unit root overconvergent $F$-isocrystals over $X_1$ are equivalent.	
\end{corollary}

\indent Beyond the corresponding unit-root situation, we should consider the corresponding non-\'etale context where we need to replace the corresponding profinite \'etale fundamental groups by the corresponding isocrytal Tannakian fundamental groups $\pi^{\mathrm{Isoc},\dagger}_1$ (see \cite[Appendix B]{DK}). Therefore the fundamental conjecture will be the following:

\begin{conjecture}
The category of the corresponding $\overline{\mathbb{Q}}_p$-representations of $\prod_{i=1}^I \pi^{\mathrm{Isoc},\dagger}_1(X_i)$ is equivalent to the corresponding category of all the $\overline{\mathbb{Q}}_p$-representations of $\pi^{\mathrm{Isoc},\dagger}_1(X/\Phi)$, which is further equivalent to the corresponding category of all the $\overline{\mathbb{Q}}_p$-overconvergent isocrystals over $X/\Phi$.
\end{conjecture}

\begin{remark}
The corresponding functoriality realizing such equivalence should be not that transparent. The reason here why one has to be sufficiently careful is that the corresponding six functors do not serve as a well-defined way to functorializing the corresponding construction and operations of the corresponding isocrystals. Instead at least for some better consideration one should look at the corresponding relative $p$-adic differential equation theory established in the scope of the	arithmetic $D^\dagger$-modules and so on.
\end{remark}

\newpage

\section{Rephrasization by Arithmetic $D^\dagger$-Modules over Stacky Disks}

\subsection{Double Points}

\indent We now rephrase some of the corresponding local pictures considered above in the framework of arithmetic $D^\dagger$-modules. The corresponding picture is obviously more transparent than the global situation. Actually we will start from the curve situation and consider then further the localization.\\

\indent The geometric framework in the local setting comes from the two pointed formal space $\mathrm{Spf}(\mathbb{Z}_p[[t]])$ and $\mathrm{Spf}W(\overline{\mathbb{F}}_p)[[t]]$. We use $s$ to denote the single closed point and we use the corresponding notation $\eta$ to denote the single open point.

\begin{setting}
\indent We now consider the corresponding setting up around the local picture in \cite{AM}. We consider the corresponding space $X_1=\mathrm{Spec}\mathbb{F}_p((t))$. Now we recall the corresponding picture of the relative $p$-adic differential equations in \cite{AM}. Recall from \cite[Section 2.1.2,2.1.3]{AM} in this case, we have the following sheaves of the corresponding differential operators with the corresponding along $t=0$ overconvergence. We use the corresponding notation
\begin{align}
D_{\mathrm{an},\mathrm{con},\mathbb{Z}_p[[t]],t=0}	
\end{align}
to denote the corresponding sheaf of analytic differential operaters along $t=0$ having overconvergence. And we use the corresponding notation:
\begin{align}
D_{\mathrm{con},\mathbb{Z}_p[[t]],t=0}	
\end{align}
to denote the corresponding sheaf of overconvergent differential operaters along $t=0$ having overconvergence. Also we have the following truncated constructions which are also very important applications. We use the corresponding notation
\begin{align}
D^{(m)}_{\mathrm{an},\mathrm{con},\mathbb{Z}_p[[t]],t=0}	
\end{align}
to denote the corresponding sheaf of analytic differential operaters along $t=0$ having overconvergence, of order $m\geq 0$. And we use the corresponding notation:
\begin{align}
D^{(m)}_{\mathrm{con},\mathbb{Z}_p[[t]],t=0}	
\end{align}
to denote the corresponding sheaf of overconvergent differential operaters along $t=0$ having overconvergence, of order $m\geq 0$. And what we will do as well is the corresponding sheaves over two points, where we will drop the corresponding notation of $t=0$.
\end{setting}

\indent And we have the following:

\begin{setting}
\indent We now consider the corresponding setting up around the local picture in \cite{AM}. We consider the corresponding space $X_1=\mathrm{Spec}\overline{\mathbb{F}}_p((t))$. Now we recall the corresponding picture of the relative $p$-adic differential equations in \cite{AM}. Recall from \cite[Section 2.1.2,2.1.3]{AM} in this case, we have the following sheaves of the corresponding differential operators with the corresponding along $t=0$ overconvergence. We use the corresponding notation
\begin{align}
D_{\mathrm{an},\mathrm{con},W(\overline{\mathbb{F}}_p)[[t]],t=0}	
\end{align}
to denote the corresponding sheaf of analytic differential operaters along $t=0$ having overconvergence. And we use the corresponding notation:
\begin{align}
D_{\mathrm{con},W(\overline{\mathbb{F}}_p)[[t]],t=0}	
\end{align}
to denote the corresponding sheaf of overconvergent differential operaters along $t=0$ having overconvergence. Also we have the following truncated constructions which are also very important applications. We use the corresponding notation
\begin{align}
D^{(m)}_{\mathrm{an},\mathrm{con},W(\overline{\mathbb{F}}_p)[[t]],t=0}	
\end{align}
to denote the corresponding sheaf of analytic differential operaters along $t=0$ having overconvergence, of order $m\geq 0$. And we use the corresponding notation:
\begin{align}
D^{(m)}_{\mathrm{con},W(\overline{\mathbb{F}}_p)[[t]],t=0}	
\end{align}
to denote the corresponding sheaf of overconvergent differential operaters along $t=0$ having overconvergence, of order $m\geq 0$. And what we will do as well is the corresponding sheaves over two points, where we will drop the corresponding notation of $t=0$.
\end{setting}

\indent Over the corresponding sheaves of rings as above we will consider those arithmetic coherent sheaves in \cite[Section 2.1.2,2.1.3,2.1.4]{AM}. And we consider the corresponding holonomic ones. Also we consider the corresponding Frobenius structures as well.

\begin{setting}
\indent We now consider the corresponding setting up around the local picture in \cite{AM}. We consider the corresponding space $X_1=\mathrm{Spec}\mathbb{F}_p((t))$. Now we recall the corresponding picture of the relative $p$-adic differential equations in \cite{AM}. Recall from \cite[Section 2.1.2,2.1.3,2.1.4]{AM} in this case, we have the following rings of the corresponding differential operators with the corresponding along $t=0$ overconvergence. We use the corresponding notation
\begin{align}
d_{\mathrm{an},\mathrm{con},\mathbb{Z}_p[[t]],t=0}	
\end{align}
to denote the corresponding ring of analytic differential operaters along $t=0$ having overconvergence. And we use the corresponding notation:
\begin{align}
d_{\mathrm{con},\mathbb{Z}_p[[t]],t=0}	
\end{align}
to denote the corresponding ring of overconvergent differential operaters along $t=0$ having overconvergence. Also we have the following truncated constructions which are also very important applications. We use the corresponding notation
\begin{align}
d^{(m)}_{\mathrm{an},\mathrm{con},\mathbb{Z}_p[[t]],t=0}	
\end{align}
to denote the corresponding ring of analytic differential operaters along $t=0$ having overconvergence, of order $m\geq 0$. And we use the corresponding notation:
\begin{align}
d^{(m)}_{\mathrm{con},\mathbb{Z}_p[[t]],t=0}	
\end{align}
to denote the corresponding ring of overconvergent differential operaters along $t=0$ having overconvergence, of order $m\geq 0$. And what we will do as well is the corresponding rings over two points, where we will drop the corresponding notation of $t=0$.
\end{setting}

\indent And we have the following:

\begin{setting}
\indent We now consider the corresponding setting up around the local picture in \cite{AM}. We consider the corresponding space $X_1=\mathrm{Spec}\overline{\mathbb{F}}_p((t))$. Now we recall the corresponding picture of the relative $p$-adic differential equations in \cite{AM}. Recall from \cite[Section 2.1.2,2.1.3,2.1.4]{AM} in this case, we have the following rings of the corresponding differential operators with the corresponding along $t=0$ overconvergence. We use the corresponding notation
\begin{align}
d_{\mathrm{an},\mathrm{con},W(\overline{\mathbb{F}}_p)[[t]],t=0}	
\end{align}
to denote the corresponding ring of analytic differential operaters along $t=0$ having overconvergence. And we use the corresponding notation:
\begin{align}
d_{\mathrm{con},W(\overline{\mathbb{F}}_p)[[t]],t=0}	
\end{align}
to denote the corresponding ring of overconvergent differential operaters along $t=0$ having overconvergence. Also we have the following truncated constructions which are also very important applications. We use the corresponding notation
\begin{align}
d^{(m)}_{\mathrm{an},\mathrm{con},W(\overline{\mathbb{F}}_p)[[t]],t=0}	
\end{align}
to denote the corresponding ring of analytic differential operaters along $t=0$ having overconvergence, of order $m\geq 0$. And we use the corresponding notation:
\begin{align}
d^{(m)}_{\mathrm{con},W(\overline{\mathbb{F}}_p)[[t]],t=0}	
\end{align}
to denote the corresponding ring of overconvergent differential operaters along $t=0$ having overconvergence, of order $m\geq 0$. And what we will do as well is the corresponding rings over two points, where we will drop the corresponding notation of $t=0$.
\end{setting}

\indent Over the corresponding rings as above we will consider those arithmetic coherent modules in \cite[Section 2.1.2,2.1.3,2.1.4]{AM}. And we consider the corresponding holonomic ones. Also we consider the corresponding Frobenius structures as well.

\begin{setting}
We now consider the following sheaves:
\begin{align}
\Pi_{\mathrm{an},\mathrm{con},\mathbb{Z}_p[[t]],t=0}:=D^{(0)}_{\mathrm{an},\mathrm{con},\mathbb{Z}_p[[t]],t=0}	
\end{align}
and 
\begin{align}
\Pi_{\mathrm{con},\mathbb{Z}_p[[t]],t=0}:=D^{(0)}_{\mathrm{con},\mathbb{Z}_p[[t]],t=0}.	
\end{align}	
Taking the global section we have:
\begin{align}
\pi_{\mathrm{an},\mathrm{con},\mathbb{Z}_p[[t]],t=0}:=d^{(0)}_{\mathrm{an},\mathrm{con},\mathbb{Z}_p[[t]],t=0}	
\end{align}
and 
\begin{align}
\pi_{\mathrm{con},\mathbb{Z}_p[[t]],t=0}:=d^{(0)}_{\mathrm{con},\mathbb{Z}_p[[t]],t=0}.	
\end{align}	
These are the corresponding Robba rings $\mathcal{R}$ and the corresponding bounded ones $\mathcal{E}^\dagger$. And what we will do as well is the corresponding rings over two points, where we will drop the corresponding notation of $t=0$.
\end{setting}

\begin{setting}
We now consider the following sheaves:
\begin{align}
\Pi_{\mathrm{an},\mathrm{con},W(\overline{\mathbb{F}}_p)[[t]],t=0}:=D^{(0)}_{\mathrm{an},\mathrm{con},W(\overline{\mathbb{F}}_p)[[t]],t=0}	
\end{align}
and 
\begin{align}
\Pi_{\mathrm{con},W(\overline{\mathbb{F}}_p)[[t]],t=0}:=D^{(0)}_{\mathrm{con},W(\overline{\mathbb{F}}_p)[[t]],t=0}.	
\end{align}	
Taking the global section we have:
\begin{align}
\pi_{\mathrm{an},\mathrm{con},W(\overline{\mathbb{F}}_p)[[t]],t=0}:=d^{(0)}_{\mathrm{an},\mathrm{con},W(\overline{\mathbb{F}}_p)[[t]],t=0}	
\end{align}
and 
\begin{align}
\pi_{\mathrm{con},W(\overline{\mathbb{F}}_p)[[t]],t=0}:=d^{(0)}_{\mathrm{con},W(\overline{\mathbb{F}}_p)[[t]],t=0}.	
\end{align}	
These are the corresponding Robba rings $\mathcal{R}$ and the corresponding bounded ones $\mathcal{E}^\dagger$. And what we will do as well is the corresponding rings over two points, where we will drop the corresponding notation of $t=0$.
\end{setting}

\indent Over these sheaves and rings we have the corresponding notion of finite free $(\varphi,\nabla)$-modules which are for instance studied very extensively in \cite{Ked8}.

\begin{proposition} \label{proposition4.7}
The corresponding category of all the finite free $(\varphi,\nabla)$-modules over $\pi_{\mathrm{an},\mathrm{con},{\mathbb{Z}}_p[[t]],t=0}$ is equivalent to the category of all the arithmetic Frobenius modules over $d_{\mathrm{an},\mathrm{con},{\mathbb{Z}}_p[[t]],t=0}$, which are assumed to be holonomic.
\end{proposition}

\begin{proof}
This is well-known, for instance see \cite[Proposition in Section 2.1.4]{AM}.	
\end{proof}

\begin{proposition} \label{proposition4.8}
The corresponding category of all the finite free $(\varphi,\nabla)$-modules over $\pi_{\mathrm{an},\mathrm{con},W(\overline{\mathbb{F}}_p)[[t]],t=0}$ is equivalent to the category of all the arithmetic Frobenius modules over $d_{\mathrm{an},\mathrm{con},W(\overline{\mathbb{F}}_p)[[t]],t=0}$, which are assumed to be holonomic.
\end{proposition}

\begin{proof}
This is well-known, for instance see \cite[Proposition in Section 2.1.4]{AM}.	
\end{proof}

\newpage
\subsection{Derived $(\varphi,\nabla)$-Modules over Stacks}

\indent In our current situation, we also consider the corresponding arithmetic differential modules over the stack $(\mathrm{Spec}\mathbb{F}_p((t)))_k/\varphi_k$ where $k=\overline{\mathbb{F}}_p$. And we also consider the corresponding $(\sigma,\nabla)$-modules over the stack $(\mathrm{Spec}\mathbb{F}_p((t)))_k/\varphi_k$ as well. Here we follow \cite{Abe} to use the corresponding framework of all the arithmetic $D$-modules over stacks.

\begin{notation}
We consider the lift $\mathbb{Z}_p[[t]]$ of the original lifting space, and we denote this $\mathrm{Spf}(\mathbb{Z}_p[[t]])$ by $\{x,\eta\}$ with $\mathrm{Spf}(\mathbb{Z}_p[[t]])_k$ by $\{\widetilde{x},\widetilde{\eta}\}$. We then have the corresponding stack $\{\widetilde{x},\widetilde{\eta}\}/\varphi_k$ with two substacks $\{\widetilde{x}\}/\varphi_k$ and $\{\widetilde{\eta}\}/\varphi_k$.	
\end{notation}

\indent \cite{Abe} defined the corresponding derived categories of arithmetic $D$-modules with the desired hearts over stacks. This is not trivial at all since we have to use the corresponding cohomological descent to realize the seemingly virtual objects through representations from schemes. In our situation we have three stacks: $\{\widetilde{x},\widetilde{\eta}\}/\varphi_k$ with two substacks $\{\widetilde{x}\}/\varphi_k$ and $\{\widetilde{\eta}\}/\varphi_k$, admitting the corresponding covering from $\{\widetilde{x},\widetilde{\eta}\}$ with two substacks $\{\widetilde{x}\}$ and $\{\widetilde{\eta}\}$:
\begin{align}
\{\widetilde{x},\widetilde{\eta}\}\rightarrow \{\widetilde{x},\widetilde{\eta}\}/\varphi_k,\\
\{\widetilde{x}\}\rightarrow \{\widetilde{x}\}/\varphi_k,\\
\{\widetilde{\eta}\}\rightarrow \{\widetilde{\eta}\}/\varphi_k.
\end{align}

\begin{definition}
We then consider the following bounded derived categories of holonomic arithmetic $D$-modules over the sheaves of overconvergent rings of differential operators over the stacks defined by Abe in \cite[Chapter 2.1]{Abe}:
\begin{align}
D^{\flat,\mathrm{holo},\varphi}_{\mathrm{con},W(\mathbb{F}_p)[[t]]_k/\varphi_k,t=0}
\end{align}
with overconvergence along $t=0$ and 
\begin{align}
D^{\flat,\mathrm{holo},\varphi}_{\mathrm{con},W(\mathbb{F}_p)[[t]]_k/\varphi_k}
\end{align}
without overconvergence along $t=0$. We use the following notations to denote the corresponding hearts:
\begin{align}
h^{\flat,\mathrm{holo},\varphi}_{\mathrm{con},W(\mathbb{F}_p)[[t]]_k/\varphi_k,t=0}
\end{align}
with overconvergence along $t=0$ and 
\begin{align}
h^{\flat,\mathrm{holo},\varphi}_{\mathrm{con},W(\mathbb{F}_p)[[t]]_k/\varphi_k}
\end{align}
without overconvergence along $t=0$.
\end{definition}

\indent We also have the corresponding derived categories over the corresponding spaces:

\begin{definition}
We then consider the following bounded derived categories of holonomic arithmetic $D$-modules over the sheaves of overconvergent rings of differential operators over the stacks defined by Abe in \cite[Chapter 2.1]{Abe}:
\begin{align}
D^{\flat,\mathrm{holo},\varphi}_{\mathrm{con},W(?)[[t]],t=0}
\end{align}
with overconvergence along $t=0$ and 
\begin{align}
D^{\flat,\mathrm{holo},\varphi}_{\mathrm{con},W(?)[[t]]}
\end{align}
without overconvergence along $t=0$. We use the following notations to denote the corresponding hearts:
\begin{align}
h^{\flat,\mathrm{holo},\varphi}_{\mathrm{con},W(?)[[t]],t=0}
\end{align}
with overconvergence along $t=0$ and 
\begin{align}
h^{\flat,\mathrm{holo},\varphi}_{\mathrm{con},W(?)[[t]]}
\end{align}
without overconvergence along $t=0$. Here $?=\mathbb{F}_p,\overline{\mathbb{F}}_p$.
\end{definition}

\indent By using the corresponding equivalence in \cref{proposition4.7} and \cref{proposition4.8} we have the following bounded derived categories with the associated hearts of the $(\varphi,\nabla)$-modules over the Robba rings:

\begin{definition}
We then consider the following bounded derived categories of $(\varphi,\nabla)$-modules over the sheaves of overconvergent rings of differential operators of order zero over the spaces:
\begin{align}
D^{\flat,\varphi,\nabla}_{\mathrm{con},W(?)[[t]],t=0}
\end{align}
with overconvergence along $t=0$ and 
\begin{align}
D^{\flat,\varphi,\nabla}_{\mathrm{con},W(?)[[t]]}
\end{align}
without overconvergence along $t=0$. We use the following notations to denote the corresponding hearts:
\begin{align}
h^{\flat,\varphi,\nabla}_{\mathrm{con},W(?)[[t]],t=0}
\end{align}
with overconvergence along $t=0$ and 
\begin{align}
h^{\flat,\varphi,\nabla}_{\mathrm{con},W(?)[[t]]}
\end{align}
without overconvergence along $t=0$. Here $?=\mathbb{F}_p,\overline{\mathbb{F}}_p$.
\end{definition}

\indent Then we define things over the key stack involved.

\begin{definition}
We now define the derived category of all the $(\varphi,\nabla)$-modules over the bounded Robba rings associated with the stack $\{\widetilde{x},\widetilde{\eta}\}$ with two substacks $\{\widetilde{x}\}$ and $\{\widetilde{\eta}\}$:
\begin{align}
\overline{D}^{\flat,\varphi,\nabla}_{\mathrm{con},W(\overline{\mathbb{F}}_p)[[t]]/\varphi_k,t=0}:=\overline{D}^\flat(\overline{h}^{\flat,\varphi,\nabla}_{\mathrm{con},W(\overline{\mathbb{F}}_p)[[t]]/\varphi_k,t=0})	
\end{align}
and 
\begin{align}
\overline{D}^{\flat,\varphi,\nabla}_{\mathrm{con},W(\overline{\mathbb{F}}_p)[[t]]/\varphi_k}:=\overline{D}^\flat(\overline{h}^{\flat,\varphi,\nabla}_{\mathrm{con},W(\overline{\mathbb{F}}_p)[[t]]/\varphi_k}).
\end{align}
Here the categories with $\overline{h}$ symbol on the right are define in the following way. First $\overline{h}^{\flat,\varphi,\nabla}_{\mathrm{con},W(\overline{\mathbb{F}}_p)[[t]]/\varphi_k,t=0}$ is defined to be a stack fibered over the stack $h^{\flat,\varphi,\nabla}_{\mathrm{con},W(\overline{\mathbb{F}}_p)[[t]],t=0}$ parametrizing the corresponding objects $(M,\varphi,\nabla,\varphi_{1},\varphi_k)$ where $(M,\varphi,\nabla)$ is parametrized by the stack below. Then we consider the derived category $D^\flat(\overline{h}^{\flat,\varphi,\nabla}_{\mathrm{con},W(\overline{\mathbb{F}}_p)[[t]]/\varphi_k,t=0})$. And we assume the corresponding cohomology groups live in the same abelian category. Each object in the corresponding derived category defined in this way will be some $(M^\bullet,\varphi,\varphi_{1},\varphi_k)$ where $M^\bullet$ is bounded complex of $(\varphi,\nabla)$-modules over the Robba ring, and $\varphi_{1},\varphi_k$ are two Frobenius morphisms such that we have:
\begin{align}
\varphi_1^*(...\rightarrow M^m \rightarrow M^{m+1}...)\overset{\sim}{\longrightarrow} (...\rightarrow M^k \rightarrow M^{k+1}...),\\
\varphi_k^*(...\rightarrow M^m \rightarrow M^{m+1}...)\overset{\sim}{\longrightarrow} (...\rightarrow M^k \rightarrow M^{k+1}...).	
\end{align}
Then we do the same construction to the rest one situation. These derived categories could be further endowed with the structure of derived stacks.
\end{definition}

\indent The corresponding arithmetic $D$-modules could be described in the same way:

\begin{definition}
We now define the derived category of all the arithmetic $D$-modules over rings of overconvergent differential operators associated with the stack $\{\widetilde{x},\widetilde{\eta}\}$ with two substacks $\{\widetilde{x}\}$ and $\{\widetilde{\eta}\}$:
\begin{align}
\overline{D}^{\flat,\mathrm{holo},\varphi}_{\mathrm{con},W(\overline{\mathbb{F}}_p)[[t]]/\varphi_k,t=0}:=\overline{D}^\flat(\overline{h}^{\flat,\mathrm{holo},\varphi}_{\mathrm{con},W(\overline{\mathbb{F}}_p)[[t]]/\varphi_k,t=0})	
\end{align}
and 
\begin{align}
\overline{D}^{\flat,\mathrm{holo},\varphi}_{\mathrm{con},W(\overline{\mathbb{F}}_p)[[t]]/\varphi_k}:=\overline{D}^\flat(\overline{h}^{\flat,\mathrm{holo},\varphi}_{\mathrm{con},W(\overline{\mathbb{F}}_p)[[t]]/\varphi_k}).
\end{align}
Here the categories with $\overline{h}$ symbol on the right are define in the following way. First $\overline{h}^{\flat,\mathrm{holo},\varphi}_{\mathrm{con},W(\overline{\mathbb{F}}_p)[[t]]/\varphi_k,t=0}$ is defined to be a stack fibered over the stack $h^{\flat,\mathrm{holo},\varphi}_{\mathrm{con},W(\overline{\mathbb{F}}_p)[[t]],t=0}$ parametrizing the corresponding objects $(M,\varphi,\varphi_{1},\varphi_k)$ where $(M,\varphi)$ is parametrized by the stack below. Then we consider the derived category $D^\flat(\overline{h}^{\flat,\mathrm{holo},\varphi}_{\mathrm{con},W(\overline{\mathbb{F}}_p)[[t]]/\varphi_k,t=0})$. And we assume the corresponding cohomology groups are coherent holonomic arithmetic $D$-modules over the same space. Each object in the corresponding derived category defined in this way will be some $(M^\bullet,\varphi,\varphi_{1},\varphi_k)$ where $M^\bullet$ is bounded complex of arithmetic $D$-modules in $D^{\flat,\mathrm{holo},\varphi}_{\mathrm{con},W(\overline{\mathbb{F}}_p)[[t]],t=0}$, and $\varphi_{1},\varphi_k$ are two Frobenius morphisms such that we have:
\begin{align}
\varphi_1^*(...\rightarrow M^m \rightarrow M^{m+1}...)\overset{\sim}{\longrightarrow} (...\rightarrow M^k \rightarrow M^{k+1}...),\\
\varphi_k^*(...\rightarrow M^m \rightarrow M^{m+1}...)\overset{\sim}{\longrightarrow} (...\rightarrow M^k \rightarrow M^{k+1}...).	
\end{align}
Then we do the same construction to the rest one situation. These derived categories could be further endowed with the structure of derived stacks.
\end{definition}

\begin{theorem}\mbox{\bf{(Kedlaya, \cite[Corollary 4.7]{Ked9})}}
The category $\overline{h}^{\flat,\varphi,\nabla}_{\mathrm{con},W(\overline{\mathbb{F}}_p)[[t]]/\varphi_k,t=0}$ is well-defined, and the category $\overline{h}^{\flat,\mathrm{holo},\varphi}_{\mathrm{con},W(\overline{\mathbb{F}}_p)[[t]]/\varphi_k,t=0}$ is well defined. The category $\overline{h}^{\flat,\varphi,\nabla}_{\mathrm{con},W(\overline{\mathbb{F}}_p)[[t]]/\varphi_k}$ is well-defined, and the category $\overline{h}^{\flat,\mathrm{holo},\varphi}_{\mathrm{con},W(\overline{\mathbb{F}}_p)[[t]]/\varphi_k}$ is well defined. To be more precise these are abelian categories.
	
\end{theorem}

\begin{proposition} \label{proposition4.15}
With the corresponding notations we have defined so far we have the following equivalence on the corresponding derived categories:
\begin{align}
\overline{D}^{\flat,\varphi,\nabla}_{\mathrm{con},W(\overline{\mathbb{F}}_p)[[t]]/\varphi_k,t=0}\overset{\sim}{\longrightarrow} \overline{D}^{\flat,\mathrm{holo},\varphi}_{\mathrm{con},W(\overline{\mathbb{F}}_p)[[t]]/\varphi_k,t=0}	
\end{align}
and 
\begin{align}
\overline{D}^{\flat,\varphi,\nabla}_{\mathrm{con},W(\overline{\mathbb{F}}_p)[[t]]/\varphi_k}\overset{\sim}{\longrightarrow} \overline{D}^{\flat,\mathrm{holo},\varphi}_{\mathrm{con},W(\overline{\mathbb{F}}_p)[[t]]/\varphi_k}.	
\end{align}
	
\end{proposition}

\newpage
\section{Polydisks and the Quotients}

\indent Now we consider the corresponding situation where we have some product of polydisks. We now choose to look at the corresponding setting up as in the following:

\begin{setting}
We now take the product of $\mathrm{Spf}\mathbb{Z}_p[[t_1]]$ and $\mathrm{Spf}\mathbb{Z}_p[[t_2]]$, as well as the product of the rigid analytic generic fibers namely the polydisks. We will use the corresponding notation $I$ to represent the set of two factors. We then have the corresponding Robba rings $\Pi_I,\Pi_I^\mathrm{bd}$ in this multivariate setting as those in \cite{PZ} and \cite{CKZ}, which carry the corresponding multivariate Frobenius structures by $\varphi_1$ and $\varphi_2$.   	
\end{setting}

\indent We have the following adic spaces:

\begin{displaymath}
\mathrm{Spa}(\pi_{\mathrm{an},\mathrm{con},\mathbb{Z}_p[[t_1]],t_1=0},\pi_{\mathrm{an},\mathrm{con},\mathbb{Z}_p[[t_1]],t_1=0}^+)\times_{\mathbb{Q}_p}\mathrm{Spa}(\pi_{\mathrm{an},\mathrm{con},\mathbb{Z}_p[[t_2]],t_2=0},\pi_{\mathrm{an},\mathrm{con},\mathbb{Z}_p[[t_2]],t_2=0}^+),	
\end{displaymath}
and 
\begin{displaymath}
\mathrm{Spa}(\pi_{\mathrm{con},\mathbb{Z}_p[[t_1]],t_1=0},\pi_{\mathrm{con},\mathbb{Z}_p[[t_1]],t_1=0}^+)\times_{\mathbb{Q}_p}\mathrm{Spa}(\pi_{\mathrm{con},\mathbb{Z}_p[[t_2]],t_2=0},\pi_{\mathrm{con},\mathbb{Z}_p[[t_2]],t_2=0}^+),
\end{displaymath}

and the corresponding $v$-sheaves:

\begin{displaymath}
\mathrm{Spa}(\pi_{\mathrm{an},\mathrm{con},\mathbb{Z}_p[[t_1]],t_1=0},\pi_{\mathrm{an},\mathrm{con},\mathbb{Z}_p[[t_1]],t_1=0}^+)/\varphi_1\times_{\mathbb{Q}_p}\mathrm{Spa}(\pi_{\mathrm{an},\mathrm{con},\mathbb{Z}_p[[t_2]],t_2=0},\pi_{\mathrm{an},\mathrm{con},\mathbb{Z}_p[[t_2]],t_2=0}^+),	
\end{displaymath}
and 
\begin{displaymath}
\mathrm{Spa}(\pi_{\mathrm{con},\mathbb{Z}_p[[t_1]],t_1=0},\pi_{\mathrm{con},\mathbb{Z}_p[[t_1]],t_1=0}^+)/\varphi_1\times_{\mathbb{Q}_p}\mathrm{Spa}(\pi_{\mathrm{con},\mathbb{Z}_p[[t_2]],t_2=0},\pi_{\mathrm{con},\mathbb{Z}_p[[t_2]],t_2=0}^+).
\end{displaymath}\\

\indent We consider the corresponding picture in the convergent unit-root situation first.

\begin{proposition}
The following categories are equivalent:\\
1. The category of all the $\varphi_I$-\'etale $\varphi_I$-modules over perfect bounded Robba ring $\widetilde{\Pi}_I^\mathrm{bd}$ (as in \cite[Section 2.3]{CKZ});\\
2. The category of all the $\mathbb{Q}_p$-representations of:
\begin{align}
\mathrm{Gal}(\widehat{\mathbb{F}_p((t_1^{1/p^\infty}))})\times \mathrm{Gal}(\widehat{\mathbb{F}_p((t_2^{1/p^\infty}))});	
\end{align}
3. The category of all the $\mathbb{Q}_p$-representations of:
\begin{align}
\mathrm{Gal}({\mathbb{F}_p((t_1^{}))})\times \mathrm{Gal}({\mathbb{F}_p((t_2^{}))});
\end{align}	
4. The category of all the $\mathbb{Q}_p$-representations of:
\begin{align}
\mathrm{Gal}(\mathrm{Spec}\mathbb{F}_p((t_1))/\varphi_1\times \mathrm{Spec}\mathbb{F}_p((t_2)));
\end{align}	
5. The category of all the $\varphi$-unit root convergent $F$-Isocrystals	carrying the corresponding actions of $\varphi_1$ and $\varphi_2$.

\end{proposition}

\begin{proof}
From 1 to 2 this is the corresponding \cite[Theorem 6.16]{CKZ}. From 2 to 3, this is trivial. From 3 to 4 this is the corresponding Drinfeld's lemma in its original form. From 4 to 5 this is following Crew-Katz \cite[Theorem 2.1]{Cr2}. 	
\end{proof}

\indent We now consider the corresponding arithmetic $D$-modules over polydisks and the corresponding quotient stacks. We first consider the corresponding arithmetic $D$-modules over the product, namely we consider the following:

\begin{setting}
We consider the corresponding ring $\Pi_I^\mathrm{bd}$ with the corresponding ring of differential operators $D_{\mathrm{con},I,t_1=t_2=0}^\mathrm{bd}$. We then have the corresponding derived category of all the holonomic arithmetic $D$-modules over $D_{\mathrm{con},I,t_1=t_2=0}^{\flat,\mathrm{holo},\varphi}$ carrying the corresponding Frobenius $\varphi$. 	
\end{setting}

\begin{remark}
Please note that the corresponding notation $t_1=t_2=0$ does not indicate the corresponding point $(0,0)$, which essentially means the larger union coming from points where $t_1=0$ or $t_2=0$.	
\end{remark}

\indent As in the corresponding situation we encountered before we consider the following derived stacks fibered over the categories considered in the previous setting.

\begin{setting}
We have the corresponding derived category of all the holonomic arithmetic $D$-modules over $D_{\mathrm{con},I,t_1=t_2=0}^{\flat,\mathrm{holo},\varphi}$ carrying the corresponding Frobenius $\varphi$. We now look at the corresponding quotient regarded as adic stack:
\begin{displaymath}
\mathrm{Sp}(\pi_{\mathrm{con},\mathbb{Z}_p[[t_1]],t_1=0})/\varphi_1\times_{\mathbb{Q}_p}\mathrm{Sp}(\pi_{\mathrm{con},\mathbb{Z}_p[[t_2]],t_2=0}).	
\end{displaymath}
Then we consider the corresponding bounded derived category of all the holonomic arithmetic $D$-modules over the stack as above, which we denote it by $\overline{D}_{\mathrm{con},I/\varphi_1,t_1=t_2=0}^{\flat,\mathrm{holo},\varphi}$ which is defined by using bounded derived category of the abelian category $\overline{h}_{\mathrm{con},I/\varphi_1,t_1=t_2=0}^{\flat,\mathrm{holo},\varphi}$ consisting of all the object taking the form of $(M,\varphi,\varphi_1,\varphi_2)$ where we have that $(M,\varphi)$ lives in the heart $h_{\mathrm{con},I,t_1=t_2=0}^{\flat,\mathrm{holo},\varphi}$.
	
\end{setting}

\begin{definition}
For any perfectoid affinoid $\mathrm{Spa}(A,A^+)$ covering the diamond $\mathrm{Sp}(\pi_{\mathrm{con},\mathbb{Z}_p[[t_1]],t_1=0})/\varphi_1$, we will consider those objects over the ring $A\widehat{\otimes}_{\mathbb{Q}_p}\pi_{\mathrm{con},\mathbb{Z}_p[[t_2]],t_2=0}$ and $A\widehat{\otimes}_{\mathbb{Q}_p}D_{\mathrm{con},\mathbb{Z}_p[[t_2]],t_2=0}$. We first define a $(\varphi_2,\nabla_2)$ modules over $A\widehat{\otimes}_{\mathbb{Q}_p}\pi_{\mathrm{con},\mathbb{Z}_p[[t_2]],t_2=0}$ to be finite projective module over $A\widehat{\otimes}_{\mathbb{Q}_p}\pi_{\mathrm{con},\mathbb{Z}_p[[t_2]],t_2=0}$ carrying the corresponding Frobenius and the corresponding connection $\nabla$   which is $A$ linear. Then we define over the diamond $\mathrm{Sp}(\pi_{\mathrm{con},\mathbb{Z}_p[[t_1]],t_1=0})/\varphi_1$ the corresponding relative $(\varphi_2,\nabla_2)$-modules which are by the natural glueing of the families of such objects defined as in the above. Here the corresponding Frobenius and $\nabla$ are required to be compatible as in the absolute situation.
\end{definition}

\begin{definition}\label{definition6.7}
For any perfectoid affinoid $\mathrm{Spa}(A,A^+)$ covering the diamond $\mathrm{Sp}(\pi_{\mathrm{an},\mathrm{con},\mathbb{Z}_p[[t_1]],t_1=0})/\varphi_1$, we will consider those objects over the ring $A\widehat{\otimes}_{\mathbb{Q}_p}\pi_{\mathrm{an,con},\mathbb{Z}_p[[t_2]],t_2=0}$ and $A\widehat{\otimes}_{\mathbb{Q}_p}D_{\mathrm{an,con},\mathbb{Z}_p[[t_2]],t_2=0}$. We first define a $(\varphi_2,\nabla_2)$ modules over $A\widehat{\otimes}_{\mathbb{Q}_p}\pi_{\mathrm{an,con},\mathbb{Z}_p[[t_2]],t_2=0}$ to be finite projective module over $A\widehat{\otimes}_{\mathbb{Q}_p}\pi_{\mathrm{an,con},\mathbb{Z}_p[[t_2]],t_2=0}$ carrying the corresponding Frobenius and the corresponding connection $\nabla$   which is $A$ linear. Then we define over the diamond $\mathrm{Sp}(\pi_{\mathrm{an},\mathrm{con},\mathbb{Z}_p[[t_1]],t_1=0})/\varphi_1$ the corresponding relative $(\varphi_2,\nabla_2)$-modules which are by the natural glueing of the families of such objects defined as in the above. Here the corresponding Frobenius and $\nabla$ are required to be compatible as in the absolute situation.
\end{definition}

\begin{definition}
Consider $\nabla_2$ and the context in the previous two definitions, we can define the corresponding generic radius of convergence with respect to each radius for the second factor relative to some perfectoid Banach affinoid $A$. After realizing the corresponding bundle as some section over some interval $[s_2,r_2]$, we define the corresponding radius $R^{\mathrm{pre}}(M,\rho_2)$ for each $\rho_2\in [s_2,r_2]$ to be:
\begin{align}
\min \{p^{-1/(p-1)}\rho_2,\varlimsup_{r\rightarrow \infty}\|{\partial_{t_2}^r}\|_{A}^{1/r}\}.	
\end{align}
Here $\partial_{t_2}=\frac{\partial}{\partial{t_2}}$. We call the module is basically solvable at $1$ relative to $A$ if we have that the following equality:
\begin{displaymath}
\lim_{\rho_2\rightarrow 1}p^{-1/(p-1)}R^\mathrm{pre}(M,\rho_2)^{-1}\rho_2^{-1}=1.	
\end{displaymath}	

\end{definition}

\indent The corresponding consideration is as in the following. This sort of observation is inspired by those due to Kedlaya \cite{Ked10} (more precise \cite[Chapter 6]{Ked10}). First by taking the quotient by $\varphi_1$ of the corresponding product adic space, starting from any corresponding overconvergent $F$-isocrystal $M$ over this adic stack, we can regard $M$ as a corresponding $(\varphi_2,\nabla_2)$-modules over the stack relative to the stack forming by the quotient of the first factor.

\begin{conjecture}
For any $\Phi$-equivariant $F$-isocrystal $(M,\varphi,\varphi_1,\varphi_2)$ in the category $\overline{h}_{\mathrm{con},I/\varphi_1,t_1=t_2=0}^{\flat,\mathrm{holo},\varphi}$. We regard this as a corresponding object taking the structure of relative (to the stack $\mathrm{Sp}(\pi_{\mathrm{con},\mathbb{Z}_p[[t_1]],t_1=0})/\varphi_1$)	$(\varphi_2,\nabla_2)$-module over the corresponding ring
\begin{center}
 $\mathrm{Sp}(\pi_{\mathrm{con},\mathbb{Z}_p[[t_1]],t_1=0})/\varphi_1\widehat{\otimes}_{\mathbb{Q}_p}\pi_{\mathrm{con},\mathbb{Z}_p[[t_2]],t_2=0}$. 
\end{center}
Then we have that with respect to this $(\varphi_2,\nabla_2)$-module structure over relative Robba ring with coefficients in the diamond \cite{Sch}, $M$ is solvable at 1 uniformly with respect each fiber $M_{x_1}$ for each $x_1$ over the stack $\mathrm{Sp}(\pi_{\mathrm{con},\mathbb{Z}_p[[t_1]],t_1=0})/\varphi_1$) namely we have $\lim_{\rho_2\rightarrow 1}R(M_{x_1},\rho_2)\rho_2^{-1}=1$ uniformly for all $x_1\in \mathrm{Sp}(\pi_{\mathrm{con},\mathbb{Z}_p[[t_1]],t_1=0})/\varphi_1$.
\end{conjecture}

\begin{remark}
Over the corresponding point $x_1$ we have that the resulting underlying $(\varphi_2,\nabla_2)$-module structure should be definitely overconvergent namely solvable at 1.	However the issue is that the corresponding family version of this condition is not automatically guaranteed.
\end{remark}

\begin{remark}
This will indicate some possible form of Drinfeld's lemma in this context, namely any $F$-isocrystal in the category $\overline{h}_{\mathrm{an,con},I/\varphi_1,t_1=t_2=0}^{\flat,\mathrm{holo},\varphi}$ is conjectured to behave as if it is overconvergent with respect the $(\varphi_2,\nabla_2)$-module structure relative to the diamond $\mathrm{Sp}(\pi_{\mathrm{an},\mathrm{con},\mathbb{Z}_p[[t_1]],t_1=0})/\varphi_1$. Moreover this should also have the analog in the perverse $p$-adic differential equation setting, and even in the setting after \cite{CKZ}, \cite{KL1} and \cite{KL2} by considering diamond coefficients \cite{Sch}, \cite{SW}. 
\end{remark}

Let us mention a little bit about the corresponding some motivation from the work \cite{Ked9} around the corresponding relative Frobenius modules and relative differential equations. Consider now the scheme $X$ we specified at the very beginning of this paper, and take $k=\overline{\mathbb{F}}_p$.

\begin{definition}
We define the category $\overline{h}^{\flat,\mathrm{holo},\varphi}_{\mathrm{con},X_k/\varphi_k}$ to be the corresponding categories fibered over $h^{\flat,\mathrm{holo},\varphi}_{\mathrm{con},X_k}$ respectively endowed with further pullbacks of $\varphi_1$ and $\varphi_k$ realizing the isomorphisms as in the previous section.	
\end{definition}

\begin{theorem}\mbox{\bf{(Kedlaya \cite[Theorem 7.3, Corollary 7.4]{Ked9})}}
For any F-isocrystal object $M$ in $\overline{h}^{\flat,\mathrm{holo},\varphi}_{\mathrm{con},X_k/\varphi_k}$ we have that there is a decomposition of $M$ with respect to the slopes for $\varphi_k$:
\begin{displaymath}
M=\bigoplus_{a/b}M_{a/b}	
\end{displaymath}
such that for each $M_{a/b}$ we have that $M_{a/b}^{\varphi_k^b-p^a}$ is an $F$-isocrystal over $X$.
\end{theorem}

\indent This means that when we have that $X$ is just the formal (punctural) disc over $\mathbb{F}_p$ as we considered above, and when we have the corresponding overconvergence along $t_1=0$ we then have that the corresponding decomposition as above of a general $F$-isocrystal object in $\overline{h}^{\flat,\mathrm{holo},\varphi}_{\mathrm{con},X_k,t_1=0}$ into:
\begin{displaymath}
\bigoplus_{a/b}M_{a/b}	
\end{displaymath}
such that for each $M_{a/b}$ we have that $M_{a/b}^{\varphi_k^b-p^a}$ is a $(\varphi_1,\nabla_1)$-module over the bounded Robba ring of $\mathbb{Q}_p$.

\indent In this current simplified and local situation we have the following proposition:

\begin{proposition}\mbox{\bf{(After Kedlaya \cite[Theorem 7.3, Corollary 7.4]{Ked9})}} \label{proposition5.14}
Consider the following two categories. The first one is the category of all the $(\varphi_1,\nabla_1)$-modules over the punctural local unit disc $\mathbb{Z}_p[[t]][1/t]$. The second one is the corresponding category of all the objects in $\overline{h}^{\flat,\varphi,\nabla}_{\mathrm{con},X_k/\varphi_k}$ taking the general form $M$ such that we have $M$ decomposes as:
\begin{displaymath}
M=M_{0}	
\end{displaymath}
in the sense discussed above, namely in the decomposition:
\begin{displaymath}
\bigoplus_{a/b}M_{a/b}	
\end{displaymath}
we will have just one component with $a/b=0$. Then we have that the two categories are equivalent.
\end{proposition}

\begin{proof}
The functor realize this equivalent is just taking the corresponding pullback along the corresponding projection $X_k\rightarrow X$. Then to show the corresponding essential surjectivity, we start from a corresponding $(\varphi_1,\varphi_k)$-object over the corresponding product space, denoted by $M_k$ satisfying the corresponding condition in the statement of the proposition. Then result follows from Kedlaya's theorem mentioned above.	
\end{proof}

\begin{corollary} \label{corollary5.15}
Let $\overline{h}^{\flat,\mathrm{holo},\varphi,\varphi_k-\mathrm{ur}}_{\mathrm{con},X_k/\varphi_k}$ denote the corresponding category of objects taking the general form $m$ in $\overline{h}^{\flat,\mathrm{holo},\varphi}_{\mathrm{con},X_k/\varphi_k}$ which as above decomposed as above (equivalently regarded as the corresponding $(\varphi_1,\nabla_1)$-object):
\begin{displaymath}
M=M_0.	
\end{displaymath}
Then we have the following equivalence:
\begin{displaymath}
\overline{h}^{\flat,\mathrm{holo},\varphi,\varphi_k-\mathrm{ur}}_{\mathrm{con},X_k/\varphi_k} \overset{\sim}{\rightarrow} \overline{h}^{\flat,\mathrm{holo},\varphi}_{\mathrm{con},X}	
\end{displaymath}
and 
\begin{displaymath}
D^\flat(\overline{h}^{\flat,\mathrm{holo},\varphi,\varphi_k-\mathrm{ur}}_{\mathrm{con},X_k/\varphi_k}) \overset{\sim}{\rightarrow} D^\flat(\overline{h}^{\flat,\mathrm{holo},\varphi}_{\mathrm{con},X}).	
\end{displaymath}	
\end{corollary}

\newpage

\section{$L$-Functions and Product Formula for Perverse $p$-adic Differential Equations over Stacks}

\indent Now we define the corresponding $L$-functions. For basic materials around the corresponding definitions around a curve, we refer closely to the work \cite{AM}.  

\begin{setting}
We work over finite level over $k=\mathbb{F}_{p^r}$, let $X$ be the corresponding curve as in the previous section.	
\end{setting}

\begin{definition} 
We define the corresponding $\varphi_k$-cohomology of any object $(M_k,\varphi,\varphi_1,\varphi_k)$ in $\overline{h}^{\flat,\mathrm{holo},\varphi}_{\mathrm{con},X_k/\varphi_k}$ by considering the corresponding hypercohomology of the following complex of holonomic $D$-modules as the definition of the $\varphi_k$-cohomology:
\[
\xymatrix@R+0pc@C+0pc{
0 \ar[r] \ar[r] \ar[r] &M_k  \ar[r]^{\varphi_k-1} \ar[r] \ar[r] &M_k \ar[r] \ar[r] \ar[r] &0
}
\]

\end{definition}

\begin{definition} 
Let $X$ be just the disc $\mathrm{Spec}(\overline{\mathbb{F}}_p[[t]][1/t])$. We define the corresponding $\varphi_k$-cohomology of any object $(M_k,\varphi,\varphi_1,\varphi_k)$ in $\overline{h}^{\flat,\mathrm{holo},\varphi}_{\mathrm{con},X_k/\varphi_k}$ by considering the corresponding hypercohomology of the following complex of holonomic $D$-modules as the definition of the $\varphi_k$-cohomology:
\[
\xymatrix@R+0pc@C+0pc{
0 \ar[r] \ar[r] \ar[r] &M_k  \ar[r]^{\varphi_k-1} \ar[r] \ar[r] &M_k \ar[r] \ar[r] \ar[r] &0
}
\]

\end{definition}

\indent We denote the corresponding complex in the above $C_{\varphi_k}^\bullet$. Now we consider the corresponding projection map $f_k:X_k\rightarrow X$ in the following development, and we will consider any geometric point $i_x:\overline{x}\rightarrow X$, and we will consider the corresponding structure map $h_k:X_k\rightarrow \mathrm{Spec}k$ with $h:X\rightarrow \mathrm{Spec}\mathbb{F}_p$.  We use the corresponding notation $i^k_x$ to denote the corresponding base change of the morphism $i_x:\overline{x}\rightarrow X$ for any $x\in |X|$.

\begin{definition} 
We define the corresponding $L$-function of any object $(M^\bullet,\varphi,\varphi_1,\varphi_k)$ in $\overline{D}^{\flat,\mathrm{holo},\varphi}_{\mathrm{con},X_k/\varphi_k}$ by using the corresponding object $(M^\bullet,\varphi)$ in $D^{\flat,\mathrm{holo},\varphi}_{\mathrm{con},X_k}$ which is denoted by $L_{X_k/\varphi_k}(M^\bullet,s)$ defined by the following formula:
\begin{align}
L_{X_k/\varphi_k}(M^\bullet,s):=\prod_i\prod_{x\in |X|}\mathrm{det}\left(1-s^{\delta_x}\varphi_1^{\delta_x}|\mathcal{H}^i(i_{x}^+f_{k,+}\mathrm{Tot}(C^\bullet_{\varphi_k}M^\bullet))\right)^{(-1)^{i+1}}. 
\end{align}	
\end{definition}

\indent Now we assume that the object in the previous definition $(M^\bullet,\varphi,\varphi_1,\varphi_k)$ satisfy the corresponding assumption:

\begin{assumption} \label{assumption6.5}
This assumption requires that $M^\bullet$ satisfy the corresponding condition which says that $M^{\bullet,\varphi_k=1}$ lives in $\overline{D}^{\flat,\mathrm{holo},\varphi}_{\mathrm{con},X}$. Certainly the base change of any such object in the latter category over $X$ satisfy the corresponding condition.
\end{assumption}

\indent Under this assumption we can now proceed to relate the corresponding $L$-function we defined to the corresponding $L$-function of the objects over $X$ after considering the corresponding local correspondence we considered in the above after Kedlaya's theorem \cite[Theorem 7.3, Corollary 7.4, Lemma 7.2]{Ked9}.

\begin{proposition}
We have the following equality:
\begin{displaymath}
L_{X_k/\varphi_k}(M^\bullet,s)= \prod_i\mathrm{det}\left(1-s\varphi_1|\mathcal{H}^i((M^\bullet)^{\varphi_k=1})\right)^{(-1)^{i+1}}.
\end{displaymath}
	
\end{proposition}

\begin{proof}
This is through the following computation. We do have locally the corresponding correspondence on the corresponding local $L$-factors under the \cref{assumption6.5} by \cref{proposition5.14} and \cref{corollary5.15}. Then we compute:
\begin{align}
L_{X_k/\varphi_k}(M^\bullet,s)&:=\prod_i\prod_{x\in |X|}\mathrm{det}\left(1-s^{\delta_x}\varphi_1^{\delta_x}|\mathcal{H}^i(i_{x}^+f_{k,+}\mathrm{Tot}(C^\bullet_{\varphi_k}M^\bullet))\right)^{(-1)^{i+1}}\\
&=	\prod_i\prod_{x\in |X|}\mathrm{det}\left(1-s^{\delta_x}\varphi_1^{\delta_x}|\mathcal{H}^i((i_{x}^+(M^\bullet))^{\varphi_k=1})\right)^{(-1)^{i+1}}\\
&=	\prod_i\mathrm{det}\left(1-s\varphi_1|\mathcal{H}^i((M^\bullet)^{\varphi_k=1})\right)^{(-1)^{i+1}}.
\end{align}

\end{proof}

\newpage
\chapter{$\infty$-Categorical Perverse $p$-adic Differential Equations over $\mathrm{Bun}_G$}

\section{Introduction}

\subsection{Motivation}

The celebrated Weil's conjecture on the Tamagawa number has been tackled in the number field case by Langlands \cite{Lan}, Kottwitz \cite{Ko} and so on. In the function field situation, Lurie and Gaitsgory \cite{GL1} invent a robust tool for studying the product formula in the level of $\infty$-category. We use the following notation to establish the corresponding introduction. First we consider $X$ a smooth proper curve defined over a finite field $\mathbb{F}_p$, and we consider $G$ a smooth group scheme affine over $X$:
\begin{displaymath}
G\rightarrow X.	
\end{displaymath}

We are going to assume that the the generic fiber of the group scheme $G$ is semisimple and simply-connected, and we assume that the fibers of the group scheme $G$ are connected. Then we use $\mathrm{Bundle}_G(X)$ to denote the smooth Artin moduli stacks of $G$-bundles over the the curve $X$, and we will for any $x\in X$  use the notation $\mathrm{Bundle}_G(\{x\})$ to denote the classifying stack of $G$-bundles at the point $x$. Then we use the notation $K_X$ to denote the function field of the curve, and for each point $x\in X$ we use the notation $K_{X,x}$ to denote the corresponding completion of the field $K_X$ at the point $x\in X$, with the corresponding ring of integer $\mathcal{O}_{X,x}$, and the corresponding residue field $\kappa(x)$. In this situation we have that the following Weil's conjecture for the global function field $K_X$ defined above:

\begin{conjecture}
In the notations above, we have the following well-defined equality:
\begin{displaymath}
\frac{|\mathrm{Bundle}_G(X)(\mathbb{F}_p)|}{q^{\mathrm{dim}\mathrm{Bundle}_G(X)}}=\prod_{x\in X} \frac{|\mathrm{Bundle}_G(\{x\})(\kappa(x))|}{q^{\mathrm{dim}\mathrm{Bundle}_G(\kappa(x))}}.	
\end{displaymath}
	
\end{conjecture}

This could be regarded as sort of local and global compatibility or a product formula. This is now a theorem due to Lurie and Gaitsgory \cite{GL1} by using $\ell$-adic Grothendieck-Lefschetz trace formulas where $\ell$ is different from the prime $p$. Actually to be more precise they proved the following formula:

\begin{theorem}\mbox{\bf (Lurie-Gaitsgory \cite[Theorem 1.4.4.1]{GL1})}\\
\begin{displaymath}
\mathrm{Tr}(\varphi^{-1}|\mathrm{H}^*(\mathrm{Bundle}_G(X\times_{\mathrm{Spec}\mathbb{F}_p} \mathrm{Spec}\overline{\mathbb{F}}_p),\mathbb{Q}_\ell))=\prod_{x\in X} \frac{|\mathrm{Bundle}_G(\{x\})(\kappa(x))|}{q^{\mathrm{dim}\mathrm{Bundle}_G(\kappa(x))}}.	
\end{displaymath}	
\end{theorem}

Actually by simple observation on the '$\ell$-independence', for instance see \cite[Section 3.2]{Abe1}, one has that actually the following corollary, where the $p$-adic cohomology theory is based on \cite{Abe1}:

\begin{corollary}
\begin{displaymath}
\mathrm{Tr}(\varphi^{-1}|\mathrm{H}^*(\mathrm{Bundle}_G(X\times_{\mathrm{Spec}\mathbb{F}_p} \mathrm{Spec}\overline{\mathbb{F}}_p),\mathbf{Dual}_p))=\prod_{x\in X} \frac{|\mathrm{Bundle}_G(\{x\})(\kappa(x))|}{q^{\mathrm{dim}\mathrm{Bundle}_G(\kappa(x))}}.	
\end{displaymath}	
\end{corollary}

Note that by directly using the cohomological language we have\footnote{$\mathbf{Dual}_p$ is the $p$-adic dualizing object as in \cite[Section 1.1.4, and above 2.2.23]{Abe1}.} :

\begin{corollary}
\begin{align}
\mathrm{Tr}(\varphi^{-1}|\mathrm{H}^*(\mathrm{Bundle}_G(&X\times_{\mathrm{Spec}\mathbb{F}_p} \mathrm{Spec}\overline{\mathbb{F}}_p),\mathbf{Dual}_p))\\
&=\prod_{x\in X} \mathrm{Tr}(\varphi^{-1}|\mathrm{H}^*(\mathrm{Bundle}_G(\{x\}\times_{\mathrm{Spec}\mathbb{F}_p)} \mathrm{Spec}\overline{\mathbb{F}}_p),\mathbf{Dual}_p)).
\end{align}
\end{corollary}

From this we have many interesting results, first the left hand is of finite type, while the right hand side is also well-defined. These are not obvious a priori at all. So one has very good understanding on the $p$-adic cohomology of $\mathrm{Bundle}_G(X)$ which is smooth but not quasi-compact. Our idea is to study the cohomology with non-trivial coefficients in the category of holonomic arithmetic $\mathcal{D}$-modules, for instance we use the notation $\mathcal{E}\in D^\mathrm{b}_{\mathrm{hol}}(X)$ to denote such an module in the $p$-adic setting. Moreover we have to assume $\mathcal{E}$ is augmented $\mathbb{E}_\infty$-ring\footnote{In fact, the sheaves of rings of differential operators whatever the forms are actually unfortunately noncommutative, therefore maybe one should really focus on the corresponding larger categories of the corresponding $A_\infty$ noncommutative algebras with the augmentation. In some sense, this is also inspired by \cite{BS}, namely the prismatic $E_\infty$-rings.}. Our main conjecture is that we have the following well defined statement:

\begin{conjecture}
\begin{align}
\mathrm{Tr}(\varphi^{-1}|\mathrm{H}^*(\mathrm{Bundle}_G(&X\times_{\mathrm{Spec}\mathbb{F}_p} \mathrm{Spec}\overline{\mathbb{F}}_p),h_{\mathrm{Bundle}_G}^*\mathcal{E}))\\
&=\prod_{x\in X} \mathrm{Tr}(\varphi^{-1}|\mathrm{H}^*(\mathrm{Bundle}_G(\{x\}\times_{\mathrm{Spec}\mathbb{F}_p)} \mathrm{Spec}\overline{\mathbb{F}}_p),h_x^*{\mathcal{E}}_x)).	
\end{align}
\end{conjecture}

Note that the cohomology groups involved for the the stack $\mathrm{Bundle}_G(\overline{X})$ is essentially nonzero in infinite degrees as in the $\ell$-adic setting in \cite{GL1} but once we consider the cotangent complex we could have more detailed description on that:

\begin{conjecture}
\begin{displaymath}
\mathrm{H}^{*}(\mathrm{cot}\mathrm{C}^*(\mathrm{Bundle}_G(X\times_{\mathrm{Spec}\mathbb{F}_p} \mathrm{Spec}\overline{\mathbb{F}}_p),h_{\mathrm{Bundle}_G}^*\mathcal{E}))\end{displaymath}
is finite dimensional over $\mathbb{Q}_p$.
\end{conjecture}

This is essentially true since we will establish the following quasi-isomorphism:
\begin{displaymath}
\mathrm{cot}\mathrm{C}^*(\mathrm{Bundle}_G(X\times_{\mathrm{Spec}\mathbb{F}_p} \mathrm{Spec}\overline{\mathbb{F}}_p),h_{\mathrm{Bundle}_G}^*\mathcal{E})\overset{\mathrm{quasi}}{\longrightarrow} \mathrm{C}^*((X\times_{\mathrm{Spec}\mathbb{F}_p} \mathrm{Spec}\overline{\mathbb{F}}_p),\mathcal{M}(G)^\mathrm{Gross}_\mathcal{E})
\end{displaymath}
where the latter is the refined rigid Gross $G$-motive with more general coefficients, which is a $p$-adic analog of Gross $G$-motive in the $\ell$-adic setting defined in \cite{Gro1}. However, this quasi-isomorphism is not obvious at all, since actually this is essentially the main thing to prove. A little bit weaker result will be the following Euler product formula for the L-functions attached to the rigid Gross $G$-motive (which we will define in our setting which is an $\infty$-$p$-adic arithmetic $\mathcal{D}$-module) which we define in the following:

\begin{definition}
We define the rigid $L$-function attached to rigid Gross $G$-motive (in general $p$-adic coefficient) as the following:
\begin{displaymath}
L_{\mathcal{M}(G),\mathrm{Frob}^-1,\mathcal{E}}(t):=\mathrm{det}(1-t\mathrm{Frob}^{-1}|\mathrm{H}^*(\overline{X},\mathcal{M}(G)_\mathcal{E})^{-1}
\end{displaymath}
\end{definition}

Then we will show that we have the following Euler product formula for this $L$-function as in the following:

\begin{align}
L_{\mathcal{M}(G),\mathrm{Frob}^-1,\mathcal{E}}(t)&:=\mathrm{det}(1-t\mathrm{Frob}^{-1}|\mathrm{H}^*(\overline{X},\mathcal{M}(G)_\mathcal{E})^{-1}\\
&=\prod_{x\in X}\mathrm{det}(1-t\mathrm{Frob}_x^{-1}|\mathrm{H}^*(\overline{X},M(G_x)_{\mathcal{E}_x})^{-1}.
\end{align}

\indent Actually if $\mathcal{E}$ is just the trivial coefficient, then the special value $L_{\mathcal{M}(G),\mathrm{Frob}^-1,\mathcal{E}}(1)$ is actually just the inverse of $\prod_{x\in X}\frac{|\mathrm{Bundle}_G(\{x\})(\kappa(x))|}{q^{\mathrm{dim}\mathrm{Bundle}_G(\kappa(x))}}$.

\newpage
\subsection{Approaches}
The main goal of us is to establish a parallel story by using $p$-adic cohomology studied by many people in the past 20 years. A well-defined at least constructible coefficient systems for at least quasi-projective schemes over $\mathbb{F}_p$ has been the central problem in the literature for a quite long time in the past, for instance people are trying to determine down the correct derived category to consider and the correct six operation to consider. The work of \cite{Ca}, \cite{Abe1} and \cite{AC1} has given some relatively good answer, which led some essential breakthrough from Kedlaya \cite{Ked11}, \cite{Ked12} on Deligne's conjecture on petits camarades \cite{De1}. The answer basically gives a well-defined t-structure in the $p$-adic setting, as well as a derived category of holonomic arithmetic $\mathcal{D}$-modules where we do have well-defined six operations. More importantly the construction works for algebraic stacks. For instance, the six operation are also defined for admissible stacks over finite fields. This allows Abe to give a proof on a $p$-adic geometric Langlands correspondence, which links the set of isocrystals and cuspidal automorphic representation of the function fields.\\

\indent In our situation, a $p$-adic story is expected. Our idea is parallel to Lurie and Gaitsgory \cite{GL1}, namely we first consider the higher categorical enrichment of the suitable derived category of arithmetic $\mathcal{D}$-modules, then define the corresponding Gross motive \cite{Gro1} attached to the group scheme $G$ (which we will call it to be \text{rigid Gross $G$-motive}). Then we will check the product formula for the Grothendieck-Lefschetz trace formula, which will finish the proof as in the $\ell$-adic proof.

\newpage
\section{Arithmetic $\mathcal{D}$-modules over Schemes}

\subsection{The $\infty$-category $\mathbb{D}^\mathrm{b}_\mathrm{hol}(X)$ of constructible objects}

\indent In this section we are going to define the $\infty$-category $\mathbb{D}^\mathrm{b}_\mathrm{hol}(X)$ as the $\infty$-categorical enrichment of the usual derived category of holonomic complexes defined for quasi-projective schemes by \cite{Abe1}. In order to be more precise we consider the following notations. First we fix a finite extension $k/\mathbb{F}_p$. And we let $X$ be a quasi-projective scheme over $k$. Then we choose some lift of $X$ over $R$ which is a complete discretely valued ring whose residue field is $k$, and we are going to use the notation $K$ to denote its fraction field. We keep the notation $X$ when we talk about its lift, if this does not cause any confusion.

\begin{definition}
We are going to systematically use the language of $\infty$-category from \cite{Lurie1}. To be more precise we first consider the derived category of the constructible holonomic $\mathcal{D}$-modules denoted by $D^\mathrm{b}_\mathrm{hol}(X)$, which play the role of the $D^\mathrm{b}_\ell(X/k)$ which is the $\ell$-adic complexes constructible over $X$ as in \cite{GL1}. Then we consider the corresponding $\infty$-categorical enrichment, which is to say we are going to construct a simplicial set from $D^\mathrm{b}_\mathrm{hol}(X)$, which we will denote it by $\mathbb{D}^\mathrm{b}_\mathrm{hol}(X)$, whose homotopy category gives rise to the derived category recalled above, which is to say $D^\mathrm{b}_\mathrm{hol}(X)$.
\end{definition}

\indent We remard that here the category $\mathrm{Hol}(X)$ defined in \cite[Definition 1.1.1]{Abe1} does not natural have enough injectives, which means one could not directly apply the construction in \cite[Section 1.3.2, Section 1.3.5]{Lurie2}. Instead we consider the following construction:

\begin{definition}
We will use the notation $\mathrm{Hol}(X)$ and $\mathrm{Con}(X)$ to denote the corresponding categories of holonomic $\mathcal{D}$-modules and the corresponding contructible ones defined in \cite[Definition after Proposition 1.3.3]{Abe1}. We use the notation $\mathrm{Ind}(\mathrm{Con}(X))$ to denote the corresponding ind-category associated to the category $\mathrm{Con}(X)$. Note that this is now a Grothendieck category which admits enough injectives, then we apply the construction in \cite[Section 1.3.2, Section 1.3.5]{Lurie2} to get the corresponding derive $\infty$-catogory, which we will denote it by $\mathbb{D}(\mathrm{Ind}(\mathrm{Con}(X)))$. Then we define:
\begin{displaymath}
\mathbb{D}_\mathrm{hol}(X):=\mathbb{D}(\mathrm{Con}(X)):=\mathbb{D}_{\mathrm{Con}(X)}(\mathrm{Ind}(\mathrm{Con}(X))),	
\end{displaymath}
where the latter is the full subcategory of $\mathbb{D}(\mathrm{Ind}(\mathrm{Con}(X)))$ consists of all the complex whose homologies live in $\mathrm{Con}(X)$. Passing to the homotopy categories we have:
\begin{displaymath}
D_\mathrm{hol}(X):=D(\mathrm{Con}(X)):=D_{\mathrm{Con}(X)}(\mathrm{Ind}(\mathrm{Con}(X))).	
\end{displaymath}

\end{definition}

\begin{remark}
By \cite[Section 1.3]{Lurie2} we have that the $\infty$-category $\mathbb{D}_\mathrm{hol}(X)$ is a stable $\infty$-category.	
\end{remark}

\begin{proof}
It suffices to prove that the cofibers are preserved under the embedding of the abelian categories $\mathrm{Con}(X)\rightarrow \mathrm{Ind}\mathrm{Con}(X)$. Then to finish, one check the long exact sequence for any mapping cone $C^\bullet(f)$, but then one finishes since on the level of derived category we have the corresponding result:
\begin{displaymath}
D(\mathrm{Con}(X))\overset{\sim}{\rightarrow}D_{\mathrm{Con}(X)}(\mathrm{Ind}\mathrm{Con}(X)).
\end{displaymath}	
\end{proof}

\subsection{Inverse and Direct Images in Derived $\infty$-Category}

\indent Recall that \cite{Abe1} finishes the searching for the six operations at least for quasi-projective schemes over $\mathbb{F}_p$, implicitly relying on Kedlaya's work on Shiho's conjecture \cite{Ked4}, \cite{Ked5}, \cite{Ked6}, \cite{Ked7}. Everything could be promoted to the derived $\infty$-categorical constructions as above. We recall everything here. To be more precise we consider any morphism of quasi-projective schemes over $k$, which we denote it as $f:X\rightarrow Y$. Then we have four derived functor:
\begin{align}
f_!: D^\mathrm{b}_\mathrm{hol}(X/K)\rightarrow D^\mathrm{b}_\mathrm{hol}(Y/K),\\
f_+: D^\mathrm{b}_\mathrm{hol}(X/K)\rightarrow D^\mathrm{b}_\mathrm{hol}(Y/K),\\
f^!: D^\mathrm{b}_\mathrm{hol}(Y/K)\rightarrow D^\mathrm{b}_\mathrm{hol}(X/K),\\
f^+: D^\mathrm{b}_\mathrm{hol}(Y/K)\rightarrow D^\mathrm{b}_\mathrm{hol}(X/K).
\end{align}

\indent We would like to promote the operations onto the derived $\infty$-level, which is to say that we have the following derived $\infty$-functors having the form as in the following:

\begin{align}
f_!: \mathbb{D}_\mathrm{hol}(X/K)\rightarrow \mathbb{D}_\mathrm{hol}(Y/K),\\
f_+: \mathbb{D}_\mathrm{hol}(X/K)\rightarrow \mathbb{D}_\mathrm{hol}(Y/K),\\
f^!: \mathbb{D}_\mathrm{hol}(Y/K)\rightarrow \mathbb{D}_\mathrm{hol}(X/K),\\
f^+: \mathbb{D}_\mathrm{hol}(Y/K)\rightarrow \mathbb{D}_\mathrm{hol}(X/K).
\end{align}

\indent Not that as in \cite{Abe1}, our construction of derived $\infty$-category used essentially the techniques of inductive objects. Therefore the six operator should be deduced from the following ones:

\begin{align}
\underline{f_!}: \mathbb{D}(\mathrm{Ind}\mathrm{Con}(X))\rightarrow \mathbb{D}(\mathrm{Ind}\mathrm{Con}(Y)),\\
\underline{f_+}: \mathbb{D}(\mathrm{Ind}\mathrm{Con}(X))\rightarrow \mathbb{D}(\mathrm{Ind}\mathrm{Con}(Y)),\\
\underline{f^!}: \mathbb{D}(\mathrm{Ind}\mathrm{Con}(Y))\rightarrow \mathbb{D}(\mathrm{Ind}\mathrm{Con}(X)),\\
\underline{f^+}: \mathbb{D}(\mathrm{Ind}\mathrm{Con}(Y))\rightarrow \mathbb{D}(\mathrm{Ind}\mathrm{Con}(X)).
\end{align}
 
which are derived from the level of ind-category $\mathrm{Ind}\mathrm{Con}(X)$. Here the definition of the derived $\infty$-category $\mathbb{D}$ is essentially \cite[Section 1.3.5]{Lurie2}. First we have the following lemma based on the discussion in \cite{Abe1}. And we following the notations (for instance the constructible t-structure). Our discussion is essentially based on the observation from \cite[Section 1.2]{Abe1}. One could actually regard the discussion presented here as sort of $\infty$-enrichment on the $t$-structures introduced in \cite{Abe1}, under the philosophy of Riemann-Hilbert correspondence.

\begin{lemma}
Let $f$ as a morphism as above, we have that the pullback $\underline{\underline{f^+}}: \mathrm{Ind}\mathrm{Con}(Y)\rightarrow \mathrm{Ind}\mathrm{Con}(X)$ is well defined which sends objects in $\mathrm{Ind}\mathrm{Con}(Y)$ to those in $\mathrm{Ind}\mathrm{Con}(X)$, and the functor is c-t-exact. 
\end{lemma}

\begin{proof}
See \cite[Section 1.3]{Abe1}.
\end{proof}

\begin{lemma}
Over the level of derived $\infty$-category, we have the right adjoint $\underline{\underline{f^\wedge_{+}}}: \mathbb{D}(\mathrm{Ind}\mathrm{Con}(X))\rightarrow \mathbb{D}(\mathrm{Ind}\mathrm{Con}(Y))$ of the functor in the previous lemma which is well defined, and which is left c-t-exact.	
\end{lemma}

\indent Then we consider the \'etale morphism:

\begin{lemma}
Let $f$ as a morphism as above that is \'etale, then we have that first the pullback functor $f^+:\mathrm{Ind}\mathrm{Con}(Y)\rightarrow \mathrm{Ind}\mathrm{Con}(X)$ which admits a left adjoint functor over the level of derived $\infty$-categories $\underline{\underline{f^\wedge_{!}}}: \mathbb{D}(\mathrm{Ind}\mathrm{Con}(X))\rightarrow \mathbb{D}(\mathrm{Ind}\mathrm{Con}(Y))$ .	
\end{lemma}

\begin{proof}
Everything extends from the level of abelian category to this setting. Also see \cite[Section 1.2.2]{Abe1}.	
\end{proof}

\indent Finally for the proper morphisms we have:

\begin{lemma}
Let $f$ as a morphism above that is proper, then the $!$-pullback functor $\underline{\underline{f^!}}:\mathbb{D}(\mathrm{Ind}\mathrm{Con}(Y))\rightarrow \mathbb{D}(\mathrm{Ind}\mathrm{Con}(X))$ is well defined.	
\end{lemma}

\newpage
\subsection{$\infty$-Ind-Categories}

\noindent In this section, we are going to define the corresponding $\infty$-Ind-Categorie which will contain essentially the constructible objects in our previous discussion. This will be a $p$-adic analog of the corresponding $\infty$-category of $\ell$-adic sheaves in \cite{GL1}. First we give the following definition:

\begin{definition}
We are going to use the notation $\overline{\mathbb{D}}_\mathrm{hol}(X/K)$	to denote the $\infty$-category $\mathrm{Ind}\mathbb{D}_\mathrm{hol}(X/K)$ constructed from the $\infty$-category $\mathbb{D}_\mathrm{hol}(X/K)$.
\end{definition}

\indent Then we discuss the extension of the operations to the category $\overline{\mathbb{D}}_\mathrm{hol}(X/K)$ as in the $\ell$-adic situation as considered in \cite{GL1}.

\indent Now base on the discussion in the previous subsection on the derived functors over the derived $\infty$-category, one could make the discussion more general to the derived categories defined above. So first the functor associated to any morphism $f:X\rightarrow Y$ between any quasi-projective schemes over $k$:

\begin{displaymath}
\underline{\underline{f^+}}: \mathbb{D}_\mathrm{hol}(Y/K)\rightarrow \mathbb{D}_\mathrm{hol}(X/K),\underline{\underline{f^\wedge_+}}: \mathbb{D}_\mathrm{hol}(X/K)\rightarrow \mathbb{D}_\mathrm{hol}(Y/K),	
\end{displaymath}

extend uniquely to the derived $\infty$-categories:

\begin{displaymath}
f^+: \overline{\mathbb{D}}_\mathrm{hol}(Y/K)\rightarrow \overline{\mathbb{D}}_\mathrm{hol}(X/K),f_+: \overline{\mathbb{D}}_\mathrm{hol}(X/K)\rightarrow \overline{\mathbb{D}}_\mathrm{hol}(Y/K).	
\end{displaymath}

\indent As above, when $f$ is \'etale we have that from the level of the constructible objects we have the following morphisms:

\begin{displaymath}
\underline{\underline{f^+}}: \mathbb{D}_\mathrm{hol}(Y/K)\rightarrow \mathbb{D}_\mathrm{hol}(X/K),\underline{\underline{f^\wedge_!}}: \mathbb{D}_\mathrm{hol}(X/K)\rightarrow \mathbb{D}_\mathrm{hol}(Y/K),	
\end{displaymath}

extend uniquely to the derived $\infty$-categories:

\begin{displaymath}
f^+: \overline{\mathbb{D}}_\mathrm{hol}(Y/K)\rightarrow \overline{\mathbb{D}}_\mathrm{hol}(X/K),f_!: \overline{\mathbb{D}}_\mathrm{hol}(X/K)\rightarrow \overline{\mathbb{D}}_\mathrm{hol}(Y/K).	
\end{displaymath}

\indent Then in the situation when $f$ is proper we have the extension of the $!$-pullback functor.

\newpage
\section{Arithmetic $\mathcal{D}$-modules over Stacks}

\indent We use a general notation $\mathcal{X}$ to denote a general algebraic stack over the base finite field essentially in \cite[Construction 3.2.5.1]{GL1}. Then we use the functor of points techniques as in \cite[Construction 3.2.5.1]{GL1} define all the categories as the corresponding homotopy limit through all the functors of points $X_i,i=1,2,...$\footnote{Taking the limit through all the functors of points by considering opposite category.}\footnote{Certainly one can consider the categories for stacks in \cite{Abe} directly and consider the corresponding inductive categories to achieve $\infty$-enhancement.}:
\begin{align}
\underset{X_i}{\mathrm{homotopylimit}}(*).	
\end{align}

\begin{definition}
We are going to systematically use the language of $\infty$-category from \cite{Lurie1}. To be more precise we first consider the derived category of the constructible holonomic $\mathcal{D}$-modules denoted by $D^\mathrm{b}_\mathrm{hol}(\mathcal{X})$, which play the role of the $D^\mathrm{b}_\ell(\mathcal{X}/k)$ which is the $\ell$-adic complexes constructible over $\mathcal{Y}$ as in \cite{GL1}. Then we consider the corresponding $\infty$-categorical enrichment, which is to say we are going to construct a simplicial set from $D^\mathrm{b}_\mathrm{hol}(\mathcal{X})$, which we will denote it by $\mathbb{D}^\mathrm{b}_\mathrm{hol}(\mathcal{X})$, whose homotopy category gives rise to the derived category recalled above, which is to say $D^\mathrm{b}_\mathrm{hol}(\mathcal{X})$.
\end{definition}

\indent We remard that here the category $\mathrm{Hol}(X)$ defined in \cite[Definition 1.1.1]{Abe1} does not natural have enough injectives, which means one could not directly apply the construction in \cite[Section 1.3.2, Section 1.3.5]{Lurie2}. Instead we consider the following construction:

\begin{definition}
We will use the notation $\mathrm{Hol}(X)$ and $\mathrm{Con}(X)$ to denote the corresponding categories of holonomic $\mathcal{D}$-modules and the corresponding contructible ones defined in \cite[Definition after Proposition 1.3.3]{Abe1}. We use the notation $\mathrm{Ind}(\mathrm{Con}(\mathcal{X}))$ to denote the corresponding ind-category associated to the category $\mathrm{Con}(\mathcal{X})$. Note that this is now a Grothendieck category which admits enough injectives, then we apply the construction in \cite[Section 1.3.2, Section 1.3.5]{Lurie2} to get the corresponding derive $\infty$-catogory, which we will denote it by $\mathbb{D}(\mathrm{Ind}(\mathrm{Con}(\mathcal{X})))$. Then we define:
\begin{displaymath}
\mathbb{D}_\mathrm{hol}(\mathcal{X}):=\mathbb{D}(\mathrm{Con}(\mathcal{X})):=\mathbb{D}_{\mathrm{Con}(\mathcal{X})}(\underset{X_i}{\mathrm{homotopylimit}}(*)\mathrm{Ind}(\mathrm{Con}(X_i))),	
\end{displaymath}
where the latter is the full subcategory of $\mathbb{D}(\mathrm{Ind}(\mathrm{Con}(\mathcal{X})))$ consists of all the complex whose homologies live in $\mathrm{Con}(\mathcal{X})$. Passing to the homotopy categories we have:
\begin{displaymath}
D_\mathrm{hol}(\mathcal{X}):=D(\mathrm{Con}(\mathcal{X})):=D_{\mathrm{Con}(\mathcal{X})}(\underset{X_i}{\mathrm{homotopylimit}}(*)\mathrm{Ind}(\mathrm{Con}(X_i))).	
\end{displaymath}

\end{definition}

\begin{remark}
By \cite[Section 1.3]{Lurie2} we have that the $\infty$-category $\mathbb{D}_\mathrm{hol}(\mathcal{X})$ is a stable $\infty$-category.	
\end{remark}

\begin{proof}
It suffices to prove that the cofibers are preserved under the embedding of the abelian categories $\mathrm{Con}(\mathcal{X})\rightarrow \mathrm{Ind}\mathrm{Con}(\mathcal{X})$. Then to finish, one check the long exact sequence for any mapping cone $C^\bullet(f)$, but then one finishes since on the level of derived category we have the corresponding result:
\begin{displaymath}
D(\mathrm{Con}(\mathcal{X}))\overset{\sim}{\rightarrow}D_{\mathrm{Con}(\mathcal{X})}(\mathrm{Ind}\mathrm{Con}(\mathcal{X})).
\end{displaymath}	
\end{proof}

\newpage
\subsection{Inverse and Direct Images in Derived $\infty$-Category}

\indent Recall that \cite[Section 1.1.3]{Abe1} finishes the searching for the six operations at least for quasi-projective schemes over $\mathbb{F}_p$, implicitly relying on Kedlaya's work on Shiho's conjecture \cite{Ked4}, \cite{Ked5}, \cite{Ked6}, \cite{Ked7}. Everything could be promoted to the derived $\infty$-categorical constructions as above. We recall everything here. To be more precise we consider any morphism, which we denote it as $f:\mathcal{X}\rightarrow \mathcal{Y}$. Then we have four derived functor:
\begin{align}
f_!: D^\mathrm{b}_\mathrm{hol}(\mathcal{X}/K)\rightarrow D^\mathrm{b}_\mathrm{hol}(\mathcal{Y}/K),\\
f_+: D^\mathrm{b}_\mathrm{hol}(\mathcal{X}/K)\rightarrow D^\mathrm{b}_\mathrm{hol}(\mathcal{Y}/K),\\
f^!: D^\mathrm{b}_\mathrm{hol}(\mathcal{Y}/K)\rightarrow D^\mathrm{b}_\mathrm{hol}(\mathcal{X}/K),\\
f^+: D^\mathrm{b}_\mathrm{hol}(\mathcal{Y}/K)\rightarrow D^\mathrm{b}_\mathrm{hol}(\mathcal{X}/K),
\end{align}
by using\footnote{All the parallel definitions up to the details by using the limit as in the following will be given in the same way.}:
\begin{align}
\underset{X_i,Y_i}{\mathrm{homotopylimit}}(*) D^\mathrm{b}_\mathrm{hol}(X_i/K)\rightarrow D^\mathrm{b}_\mathrm{hol}(Y_i/K),\\
\underset{X_i,Y_i}{\mathrm{homotopylimit}}(*) D^\mathrm{b}_\mathrm{hol}(X_i/K)\rightarrow D^\mathrm{b}_\mathrm{hol}(Y_i/K),\\
\underset{X_i,Y_i}{\mathrm{homotopylimit}}(*) D^\mathrm{b}_\mathrm{hol}(Y_i/K)\rightarrow D^\mathrm{b}_\mathrm{hol}(X_i/K),\\
\underset{X_i,Y_i}{\mathrm{homotopylimit}}(*) D^\mathrm{b}_\mathrm{hol}(Y_i/K)\rightarrow D^\mathrm{b}_\mathrm{hol}(X_i/K).
\end{align}

\indent We would like to promote the operations onto the derived $\infty$-level, which is to say that we have the following derived $\infty$-functors having the form as in the following:

\begin{align}
f_!: \mathbb{D}_\mathrm{hol}(\mathcal{X}/K)\rightarrow \mathbb{D}_\mathrm{hol}(\mathcal{Y}/K),\\
f_+: \mathbb{D}_\mathrm{hol}(\mathcal{X}/K)\rightarrow \mathbb{D}_\mathrm{hol}(\mathcal{Y}/K),\\
f^!: \mathbb{D}_\mathrm{hol}(\mathcal{Y}/K)\rightarrow \mathbb{D}_\mathrm{hol}(\mathcal{X}/K),\\
f^+: \mathbb{D}_\mathrm{hol}(\mathcal{Y}/K)\rightarrow \mathbb{D}_\mathrm{hol}(\mathcal{X}/K).
\end{align}

\indent Not that as in \cite[Section 1.2]{Abe1}, our construction of derived $\infty$-category used essentially the techniques of inductive objects. Therefore the six operator should be deduced from the following ones:

\begin{align}
\underline{f_!}: \mathbb{D}(\mathrm{Ind}\mathrm{Con}(\mathcal{X}))\rightarrow \mathbb{D}(\mathrm{Ind}\mathrm{Con}(\mathcal{Y})),\\
\underline{f_+}: \mathbb{D}(\mathrm{Ind}\mathrm{Con}(\mathcal{X}))\rightarrow \mathbb{D}(\mathrm{Ind}\mathrm{Con}(\mathcal{Y})),\\
\underline{f^!}: \mathbb{D}(\mathrm{Ind}\mathrm{Con}(\mathcal{Y}))\rightarrow \mathbb{D}(\mathrm{Ind}\mathrm{Con}(\mathcal{X})),\\
\underline{f^+}: \mathbb{D}(\mathrm{Ind}\mathrm{Con}(\mathcal{Y}))\rightarrow \mathbb{D}(\mathrm{Ind}\mathrm{Con}(\mathcal{X})).
\end{align}
 
which are derived from the level of ind-category $\mathrm{Ind}\mathrm{Con}(\mathcal{X})$. Here the definition of the derived $\infty$-category $\mathbb{D}$ is essentially \cite[Section 1.3.5]{Lurie2}. First we have the following lemma based on the discussion in \cite{Abe1}. And we following the notations (for instance the constructible t-structure). Our discussion is essentially based on the observation from \cite[Section 1.2]{Abe1}. One could actually regard the discussion presented here as sort of $\infty$-enrichment on the $t$-structures introduced in \cite{Abe1}, under the philosophy of Riemann-Hilbert correspondence.

\begin{lemma}
Let $f$ as a morphism as above, we have that the pullback $\underline{\underline{f^+}}: \mathrm{Ind}\mathrm{Con}(\mathcal{Y})\rightarrow \mathrm{Ind}\mathrm{Con}(\mathcal{X})$ is well defined which sends objects in $\mathrm{Ind}\mathrm{Con}(\mathcal{Y})$ to those in $\mathrm{Ind}\mathrm{Con}(\mathcal{X})$, and the functor is c-t-exact. 
\end{lemma}

\begin{proof}
See \cite[Section 1.3]{Abe1}.
\end{proof}

\begin{lemma}
Over the level of derived $\infty$-category, we have the right adjoint $\underline{\underline{f^\wedge_{+}}}: \mathbb{D}(\mathrm{Ind}\mathrm{Con}(\mathcal{X}))\rightarrow \mathbb{D}(\mathrm{Ind}\mathrm{Con}(\mathcal{Y}))$ of the functor in the previous lemma which is well defined, and which is left c-t-exact.	
\end{lemma}

\indent Then we consider the \'etale morphism:

\begin{lemma}
Let $f$ as a morphism as above that is \'etale, then we have that first the pullback functor $f^+:\mathrm{Ind}\mathrm{Con}(\mathcal{Y})\rightarrow \mathrm{Ind}\mathrm{Con}(\mathcal{X})$ which admits a left adjoint functor over the level of derived $\infty$-categories $\underline{\underline{f^\wedge_{!}}}: \mathbb{D}(\mathrm{Ind}\mathrm{Con}(\mathcal{X}))\rightarrow \mathbb{D}(\mathrm{Ind}\mathrm{Con}(\mathcal{Y}))$ .	
\end{lemma}

\begin{proof}
Everything extends from the level of abelian category to this setting. Also see \cite[1.2.2]{Abe1}.	
\end{proof}

\indent Finally for the proper morphisms we have:

\begin{lemma}
Let $f$ as a morphism above that is proper, then the $!$-pullback functor $\underline{\underline{f^!}}:\mathbb{D}(\mathrm{Ind}\mathrm{Con}(\mathcal{Y}))\rightarrow \mathbb{D}(\mathrm{Ind}\mathrm{Con}(\mathcal{X}))$ is well defined.	
\end{lemma}

\newpage
\section{Arithmetic of Rigid Gross $G$-Motives: $\mathrm{Bun}_G$}

\indent We now define the corresponding rigid Gross $G$-motives as in \cite{Gro1} and \cite{GL1} in the corresponding $\ell$-adic situation. In the analogy of the corresponding situation in \cite{GL1} we use the language of the corresponding arithmetic $D$-modules. The relatively different aspect is that we have to focus on the corresponding $E_\infty$-algebras with the augmentation\footnote{In fact, the sheaves of rings of differential operators whatever the forms are actually unfortunately noncommutative, therefore maybe one should really focus on the corresponding larger categories of the corresponding $A_\infty$-algebras with the augmentation.}. Our idea is to study the cohomology with non-trivial coefficients in the category of holonomic arithmetic $\mathcal{D}$-modules, for instance we use the notation $\mathcal{E}\in \mathbb{D}_{\mathrm{hol}}(X)$ to denote such an module in the $p$-adic setting. Our main conjecture is that we have the following well defined statement:

\begin{conjecture}
\begin{align}
\mathrm{Tr}(\varphi^{-1}|\mathrm{H}^*(\mathrm{Bundle}_G(&X\times_{\mathrm{Spec}\mathbb{F}_p} \mathrm{Spec}\overline{\mathbb{F}}_p),\mathrm{pr}_*h_{\mathrm{Bundle}_G}^*\mathcal{E}))\\
&=\prod_{x\in X} \mathrm{Tr}(\varphi^{-1}|\mathrm{H}^*(\mathrm{Bundle}_G(\{x\}\times_{\mathrm{Spec}\mathbb{F}_p)} \mathrm{Spec}\overline{\mathbb{F}}_p),h_x^*{\mathcal{E}}_x)).	
\end{align}
\end{conjecture}

The corresponding cohomology groups and rings are defined in the following:

\begin{definition}
As in \cite[Construction 3.2.5.1]{GL1}, we define the following cohomology groups associated to the corresponding arithmetic $D$-module in the $\infty$-category involved by taking throughout all the functors of the corresponding points associated withe involving stacks:
\begin{align}
\mathrm{H}^*(\mathrm{Bundle}_G(&X\times_{\mathrm{Spec}\mathbb{F}_p} \mathrm{Spec}\overline{\mathbb{F}}_p),\mathrm{pr}_*h_{\mathrm{Bundle}_G}^*\mathcal{E})\\
&:=\underset{X_i}{\mathrm{homotopylimit}}(*)\mathrm{H}^*(X_{i,\mathrm{Bundle}_G(X\times_{\mathrm{Spec}\mathbb{F}_p} \mathrm{Spec}\overline{\mathbb{F}}_p)},\mathrm{pr}_*h_{\mathrm{Bundle}_G}^*\mathcal{E})
\end{align}
where we have the following notations of morphisms:
\begin{align}
h_{\mathrm{Bundle}_G}:\mathrm{Bundle}_G(X)\times X\rightarrow \mathrm{Classifyingstack}_G\rightarrow X,\\
\mathrm{pr}:\mathrm{Bundle}_G(X)\times X\rightarrow \mathrm{Bundle}_G(X),\\
\mathrm{pr}': \mathrm{Bundle}_G(X)\times X\rightarrow X.	
\end{align}
	
\end{definition}

\begin{definition}
As in \cite[Construction 3.2.5.1]{GL1}, we define the following cohomology groups associated to the corresponding arithmetic $D$-module in the $\infty$-category involved by taking throughout all the functors of the corresponding points associated withe involving stacks:
\begin{align}
\mathrm{H}^*(\mathrm{Bundle}_G(&\{x\}\times_{\mathrm{Spec}\mathbb{F}_p)} \mathrm{Spec}\overline{\mathbb{F}}_p),h_x^*{\mathcal{E}}_x)\\
&:=\underset{X_i}{\mathrm{homotopylimit}}(*)\mathrm{H}^*(X_{i,\mathrm{Bundle}_G(\{x\}\times_{\mathrm{Spec}\mathbb{F}_p)} \mathrm{Spec}\overline{\mathbb{F}}_p)},h_x^*{\mathcal{E}}_x)
\end{align}
where we have the following notations of morphisms\footnote{Certainly up to taking the corresponding base change to the corresponding algebraically closed field.}:
\begin{align}
h:\mathrm{Classifyingstack}_G\rightarrow X,\\
h_x:\mathrm{Bundle}_G(\{x\})\rightarrow \{x\}.\\	
\end{align}
	
\end{definition}

Note that the cohomology groups involved for the the stack $\mathrm{Bundle}_G(\overline{X})$ is essentially nonzero in infinite degrees as in the $\ell$-adic setting in \cite{GL1} but once we consider the cotangent complex we could have more detailed description on that:

\begin{conjecture}
We now here assume that the object $\mathcal{E}$ is an augmented commutative $\mathbb{E}_\infty$-algebra.
\begin{displaymath}
\mathrm{H}^{*}(\mathrm{cot}\mathrm{C}^*(\mathrm{Bundle}_G(X\times_{\mathrm{Spec}\mathbb{F}_p} \mathrm{Spec}\overline{\mathbb{F}}_p),\mathrm{pr}_*h_{\mathrm{Bundle}_G}^*\mathcal{E}))
\end{displaymath}
is finite dimensional over $\mathbb{Q}_p$.
\end{conjecture}

This is essentially true since we will establish the following quasi-isomorphism:
\begin{displaymath}
\mathrm{cot}\mathrm{C}^*(\mathrm{Bundle}_G(X\times_{\mathrm{Spec}\mathbb{F}_p} \mathrm{Spec}\overline{\mathbb{F}}_p),h_{\mathrm{Bundle}_G}^*\mathcal{E})\overset{\mathrm{quasi}}{\longrightarrow} \mathrm{C}^*((X\times_{\mathrm{Spec}\mathbb{F}_p} \mathrm{Spec}\overline{\mathbb{F}}_p),\mathcal{M}(G)^\mathrm{Gross}_\mathcal{E})
\end{displaymath}
where the latter is the refined rigid Gross $G$-motive with more general coefficients which is a $p$-adic analog of Gross $G$-motive in the $\ell$-adic setting defined in \cite{Gro1}. However, this quasi-isomorphism is not obvious at all, since actually this is essentially the main thing to prove.

\begin{definition}
We now here assume that the object $\mathcal{E}$ is an augmented commutative $\mathbb{E}_\infty$-algebra.
\begin{align}
\mathcal{M}(G)^\mathrm{Gross}_\mathcal{E}:= \mathrm{cotangentfibre}(h_*h^*\mathcal{E}).	
\end{align}
where $h: \mathrm{BC}_{\overline{k}}\mathrm{Classifyingstack}_G \rightarrow \mathrm{BC}_{\overline{k}} X$.
\end{definition}

 A little bit weaker result will be the following Euler product formula for the L-functions attached to the rigid Gross $G$-motive (which we will define in our setting which is an $\infty$-$p$-adic arithmetic $\mathcal{D}$-module) which we define in the following:

\begin{definition}
We now here assume that the object $\mathcal{E}$ is an augmented commutative $\mathbb{E}_\infty$-algebra. We define the rigid $L$-function attached to rigid Gross $G$-motive (in general $p$-adic coefficient) as the following:
\begin{displaymath}
L_{\mathcal{M}(G),\mathrm{Frob}^{-1},\mathcal{E}}(t):=\mathrm{det}(1-t\mathrm{Frob}^{-1}|\mathrm{H}^*(\overline{X},\mathcal{M}(G)_\mathcal{E})^{-1}
\end{displaymath}
\end{definition}

Then we will show that we have the following Euler product formula for this $L$-function as in the following:

\begin{align}
L_{\mathcal{M}(G),\mathrm{Frob}^-1,\mathcal{E}}(t)&:=\mathrm{det}(1-t\mathrm{Frob}^{-1}|\mathrm{H}^*(\overline{X},\mathcal{M}(G)_\mathcal{E})^{-1}\\
&=\prod_{x\in X}\mathrm{det}(1-t\mathrm{Frob}_x^{-1}|\mathrm{H}^*(\overline{X},\mathcal{M}(G_x)_{\mathcal{E}_x})^{-1}.
\end{align}

\indent Actually if $\mathcal{E}$ is just the trivial coefficient, then the special value $L_{\mathcal{M}(G),\mathrm{Frob}^{-1},\mathcal{E}}(1)$ is actually just the inverse of $\prod_{x\in X}\frac{|\mathrm{Bundle}_G(\{x\})(\kappa(x))|}{q^{\mathrm{dim}\mathrm{Bundle}_G(\kappa(x))}}$.\\

\indent We now look at some results as consequences of \cite{De1}, \cite{Ked11} and \cite{Ked12}. Namely we now assume $\mathcal{E}$ is some isocrystal $E$:

\begin{proposition}
\begin{align}
\mathrm{Tr}(\varphi^{-1}|\mathrm{H}^*(\mathrm{Bundle}_G(&X\times_{\mathrm{Spec}\mathbb{F}_p} \mathrm{Spec}\overline{\mathbb{F}}_p),\mathrm{pr}_*h_{\mathrm{Bundle}_G}^*E))\\
&=\prod_{x\in X} \mathrm{Tr}(\varphi^{-1}|\mathrm{H}^*(\mathrm{Bundle}_G(\{x\}\times_{\mathrm{Spec}\mathbb{F}_p)} \mathrm{Spec}\overline{\mathbb{F}}_p),h_x^*{E}_x)).	
\end{align}
\end{proposition}

\begin{proof}
Since now $E$ is some isocrystal, therefore one can apply results of conjectures of petits camarades from \cite[Conjecture 1.2.10]{De1} as proved in \cite{Ked11} and \cite[Theorem 0.1.1, Theorem 0.1.2]{Ked12}. Then one reduces the corresponding proposition to the $\ell$-adic situation as in \cite[Theorem 1.4.4.1]{GL1}.	
\end{proof}

\begin{proposition}
We now here assume that the object $E$ is an augmented commutative $\mathbb{E}_\infty$-algebra. Then we have the following product formula:
\begin{align}
L_{\mathcal{M}(G),\mathrm{Frob}^{-1},\mathcal{E}}(t)&:=\mathrm{det}(1-t\mathrm{Frob}^{-1}|\mathrm{H}^*(\overline{X},\mathcal{M}(G)_\mathcal{E})^{-1}\\
&=\prod_{x\in X}\mathrm{det}(1-t\mathrm{Frob}_x^{-1}|\mathrm{H}^*(\overline{X},\mathcal{M}(G_x)_{\mathcal{E}_x})^{-1}.
\end{align}
\end{proposition}

\begin{proof}
Since now $E$ is some isocrystal, therefore one can apply results of conjectures of petits camarades from \cite[Conjecture 1.2.10]{De1} as proved in \cite{Ked11} and \cite[Theorem 0.1.1, Theorem 0.1.2]{Ked12}. Then one reduces the corresponding proposition to the $\ell$-adic situation as in \cite[Theorem 4.5.3.1]{GL1}.	
\end{proof}

\newpage
\section{Arithmetic of Rigid Gross $G$-Motives: $\mathrm{Bun}^\mathrm{Parabolic,P}_G$}

\indent We now define the corresponding rigid Gross $G$-motives as in \cite{Gro1} and \cite{GL1} in the corresponding $\ell$-adic situation. And we will consider the corresponding context of moduli of parabolic $G$-bundles as in \cite[Section 5.5.1]{GL1}. Therefore we will fix some parabolic $P$ as in \cite[Section 5.5.1]{GL1} which is denoted by $P_0$\footnote{Suitable inner form $G$ namely the corresponding adjoint semisimple one will give us the chance to reduce everything to the previous section by considering $P$-bundles as in \cite[Example 5.5.1.8]{GL1}.}. In the analogy of the corresponding situation in \cite{GL1} we use the language of the corresponding arithmetic $D$-modules. The relatively different aspect is that we have to focus on the corresponding $E_\infty$-algebras with the augmentation\footnote{In fact, the sheaves of rings of differential operators whatever the forms are actually unfortunately noncommutative, therefore maybe one should really focus on the corresponding larger categories of the corresponding $A_\infty$-algebras with the augmentation.}. Our idea is to study the cohomology with non-trivial coefficients in the category of holonomic arithmetic $\mathcal{D}$-modules, for instance we use the notation $\mathcal{E}\in \mathbb{D}_{\mathrm{hol}}(X)$ to denote such an module in the $p$-adic setting. Our main conjecture is that we have the following well defined statement:

\begin{conjecture}
\begin{align}
\mathrm{Tr}(\varphi^{-1}|\mathrm{H}^*(\mathrm{Bundle}^\mathrm{Parabolic,P}_G(X\times_{\mathrm{Spec}\mathbb{F}_p} \mathrm{Spec}\overline{\mathbb{F}}_p),\mathrm{pr}_*h_{\mathrm{Bundle}^\mathrm{Parabolic,P}_G}^*\mathcal{E}))\\
=\prod_{x\in X} \mathrm{Tr}(\varphi^{-1}|\mathrm{H}^*(\mathrm{Bundle}^\mathrm{Parabolic,P}_G(\{x\}\times_{\mathrm{Spec}\mathbb{F}_p)} \mathrm{Spec}\overline{\mathbb{F}}_p),h_x^*{\mathcal{E}}_x)).	
\end{align}
\end{conjecture}

The corresponding cohomology groups and rings are defined in the following:

\begin{definition}
As in \cite[Construction 3.2.5.1]{GL1}, we define the following cohomology groups associated to the corresponding arithmetic $D$-module in the $\infty$-category involved by taking throughout all the functors of the corresponding points associated withe involving stacks:
\begin{align}
\mathrm{H}^*(\mathrm{Bundle}^\mathrm{Parabolic,P}_G(&X\times_{\mathrm{Spec}\mathbb{F}_p} \mathrm{Spec}\overline{\mathbb{F}}_p),\mathrm{pr}_*h_{\mathrm{Bundle}^\mathrm{Parabolic,P}_G}^*\mathcal{E})\\
&:=\underset{X_i}{\mathrm{homotopylimit}}(*)\mathrm{H}^*(X_{i,\mathrm{Bundle}^\mathrm{Parabolic,P}_G(X\times_{\mathrm{Spec}\mathbb{F}_p} \mathrm{Spec}\overline{\mathbb{F}}_p)},\mathrm{pr}_*h_{\mathrm{Bundle}^\mathrm{Parabolic,P}_G}^*\mathcal{E})
\end{align}
where we have the following notations of morphisms:
\begin{align}
h_{\mathrm{Bundle}^\mathrm{Parabolic,P}_G}:\mathrm{Bundle}^\mathrm{Parabolic,P}_G(X)\times X\rightarrow \mathrm{Classifyingstack}^P_G\rightarrow X,\\
\mathrm{pr}:\mathrm{Bundle}^\mathrm{Parabolic,P}_G(X)\times X\rightarrow \mathrm{Bundle}^\mathrm{Parabolic,P}_G(X),\\
\mathrm{pr}': \mathrm{Bundle}^\mathrm{Parabolic,P}_G(X)\times X\rightarrow X.	
\end{align}
	
\end{definition}

\begin{definition}
As in \cite[Construction 3.2.5.1]{GL1}, we define the following cohomology groups associated to the corresponding arithmetic $D$-module in the $\infty$-category involved by taking throughout all the functors of the corresponding points associated withe involving stacks:
\begin{align}
\mathrm{H}^*(\mathrm{Bundle}^\mathrm{Parabolic,P}_G(&\{x\}\times_{\mathrm{Spec}\mathbb{F}_p)} \mathrm{Spec}\overline{\mathbb{F}}_p),h_x^*{\mathcal{E}}_x)\\
&:=\underset{X_i}{\mathrm{homotopylimit}}(*)\mathrm{H}^*(X_{i,\mathrm{Bundle}^\mathrm{Parabolic,P}_G(\{x\}\times_{\mathrm{Spec}\mathbb{F}_p)} \mathrm{Spec}\overline{\mathbb{F}}_p)},h_x^*{\mathcal{E}}_x)
\end{align}
where we have the following notations of morphisms\footnote{Certainly up to taking the corresponding base change to the corresponding algebraically closed field.}:
\begin{align}
h:\mathrm{Classifyingstack}^P_G\rightarrow X,\\
h_x:\mathrm{Bundle}^\mathrm{Parabolic,P}_G(\{x\})\rightarrow \{x\}.\\	
\end{align}
	
\end{definition}

Note that the cohomology groups involved for the the stack $\mathrm{Bundle}^\mathrm{Parabolic,P}_G(\overline{X})$ is essentially nonzero in infinite degrees as in the $\ell$-adic setting in \cite{GL1} but once we consider the cotangent complex we could have more detailed description on that:

\begin{conjecture}
We now here assume that the object $\mathcal{E}$ is an augmented commutative $\mathbb{E}_\infty$-algebra.
\begin{displaymath}
\mathrm{H}^{*}(\mathrm{cot}\mathrm{C}^*(\mathrm{Bundle}^\mathrm{Parabolic,P}_G(X\times_{\mathrm{Spec}\mathbb{F}_p} \mathrm{Spec}\overline{\mathbb{F}}_p),\mathrm{pr}_*h_{\mathrm{Bundle}^\mathrm{Parabolic,P}_G}^*\mathcal{E}))
\end{displaymath}
is finite dimensional over $\mathbb{Q}_p$.
\end{conjecture}

This is essentially true since we will establish the following quasi-isomorphism:
\begin{align}
\mathrm{cot}\mathrm{C}^*(\mathrm{Bundle}^\mathrm{Parabolic,P}_G(X\times_{\mathrm{Spec}\mathbb{F}_p} \mathrm{Spec}\overline{\mathbb{F}}_p),h_{\mathrm{Bundle}^\mathrm{Parabolic,P}_G}^*\mathcal{E})\overset{\mathrm{quasi}}{\longrightarrow} \mathrm{C}^*((X\times_{\mathrm{Spec}\mathbb{F}_p} \mathrm{Spec}\overline{\mathbb{F}}_p),\mathcal{M}(G)^\mathrm{Gross,Parabolic}_\mathcal{E})
\end{align}
where the latter is the refined rigid Gross $G$-motive with more general coefficients which is a $p$-adic analog of Gross $G$-motive in the $\ell$-adic setting defined in \cite{Gro1}. However, this quasi-isomorphism is not obvious at all, since actually this is essentially the main thing to prove.

\begin{definition}
We now here assume that the object $\mathcal{E}$ is an augmented commutative $\mathbb{E}_\infty$-algebra.
\begin{align}
\mathcal{M}(G)^\mathrm{Gross,Parabolic}_\mathcal{E}:= \mathrm{cotangentfibre}(h_*h^*\mathcal{E}).	
\end{align}
where $h: \mathrm{BC}_{\overline{k}}\mathrm{Classifyingstack}^P_G \rightarrow \mathrm{BC}_{\overline{k}} X$.
\end{definition}

 A little bit weaker result will be the following Euler product formula for the L-functions attached to the rigid Gross $G$-motive (which we will define in our setting which is an $\infty$-$p$-adic arithmetic $\mathcal{D}$-module) which we define in the following:

\begin{definition}
We now here assume that the object $\mathcal{E}$ is an augmented commutative $\mathbb{E}_\infty$-algebra. We define the rigid $L$-function attached to rigid Gross $G$-motive (in general $p$-adic coefficient) as the following:
\begin{displaymath}
L_{\mathcal{M}(G)^\mathrm{Gross,Parabolic},\mathrm{Frob}^{-1},\mathcal{E}}(t):=\mathrm{det}(1-t\mathrm{Frob}^{-1}|\mathrm{H}^*(\overline{X},\mathcal{M}(G)^\mathrm{Gross,Parabolic}_\mathcal{E})^{-1}
\end{displaymath}
\end{definition}

\begin{conjecture}
We now here assume that the object $\mathcal{E}$ is an augmented commutative $\mathbb{E}_\infty$-algebra.
\begin{align}
L_{\mathcal{M}(G),\mathrm{Frob}^-1,\mathcal{E}}(t)&:=\mathrm{det}(1-t\mathrm{Frob}^{-1}|\mathrm{H}^*(\overline{X},\mathcal{M}(G)^\mathrm{Gross,Parabolic}_\mathcal{E})^{-1}\\
&=\prod_{x\in X}\mathrm{det}(1-t\mathrm{Frob}_x^{-1}|\mathrm{H}^*(\overline{X},\mathcal{M}(G_x)^\mathrm{Gross,Parabolic}_{\mathcal{E}_x})^{-1}.
\end{align}

\end{conjecture}

\newpage
\section{Arithmetic of Rigid Gross $G$-Motives: $\mathrm{Bun}^\mathrm{Local}_{G(K_{X,x})}$}

\indent We now take a look at some local version. Here the interesting thing is that we will have $\mathrm{Bun}_G$ over Fargues-Fontaine curves in equal characteristic (namely the moduli of $t$-motivic $G$-bundles) as in \cite{FF}, \cite{FS}, \cite{GL2}, \cite{HP}, \cite{KL1}, \cite{KL2}, \cite{SW}, around some point $x\in X$. And we fix a local field $K_{X,x}$ coming from some point over $X$. Then for any $t_x$-adic Lie group $G(K_{X,x})$ we define the following rigid local Gross $G$-motive. Therefore now the object $\mathcal{E}$ will be over the Fargues-Fontaine curve $\mathrm{FF}_{\widehat{\mathcal{O}}_{X,x}}$\footnote{As in \cite{GL2}, which should be related to Scholze's definition through $v$-stacks as in \cite{FS}, \cite{SW} and \cite{Sch} and should be related to the construction in \cite{KL2}. Here we work over categories of formal schemes.} around some neighbourhood $\mathrm{Sp}(\widehat{\mathcal{O}}_{X,x})$ of $x$. Again we follow \cite{GL1} closely.

\begin{remark}
We regard the stack $\mathrm{FF}_{\widehat{\mathcal{O}}_{X,x}}$ as some formal stack by taking presentation from some formal scheme $Y$ to the quotient through Frobenius. Then we define the classifying stack $\mathrm{Classifyingstack}_{G(K_{X,x}),\mathrm{FF}_{\widehat{\mathcal{O}}_{X,x}}}$ over $\mathrm{FF}_{\widehat{\mathcal{O}}_{X,x}}$ by taking the corresponding quotient by the group $G(K_{X,x})$.	
\end{remark}

\begin{definition}
And we define $\mathrm{Bun}^\mathrm{Local}_{G(K_{X,x}),X,x}$ as the corresponding formal stack of $G(K_{X,x})$-bundles over the disc $\mathrm{Sp}(\widehat{\mathcal{O}}_{X,x})$. And we define $\mathrm{Bun}^\mathrm{Local}_{G(K_{X,x}),\mathrm{FF}_{\widehat{\mathcal{O}}_{X,x}}}$ as the corresponding formal stack of $G(K_{X,x})$-bundles over $\mathrm{FF}_{\widehat{\mathcal{O}}_{X,x}}$.
\end{definition}

\begin{definition}
We now here assume that the object $\mathcal{E}$ is an augmented commutative $\mathbb{E}_\infty$-algebra over $\mathrm{FF}_{\widehat{\mathcal{O}}_{X,x}}$.
\begin{align}
\mathcal{M}(G)^\mathrm{Gross,Local}_\mathcal{E}:= \mathrm{cotangentfibre}(h_*h^*\mathcal{E}).	
\end{align}
where $h: \mathrm{BC}_{\overline{k}}\mathrm{Classifyingstack}_{G(K_{X,x}),\mathrm{FF}_{\widehat{\mathcal{O}}_{X,x}}} \rightarrow \mathrm{BC}_{\overline{k}}\mathrm{FF}_{\widehat{\mathcal{O}}_{X,x}}$.
\end{definition}

 A little bit weaker result will be the following Euler product formula for the L-functions attached to the rigid Gross $G$-motive (which we will define in our setting which is an $\infty$-$p$-adic arithmetic $\mathcal{D}$-module) which we define in the following:

\begin{definition}
We now here assume that the object $\mathcal{E}$ is an augmented commutative $\mathbb{E}_\infty$-algebra. We define the rigid $L$-factor attached to rigid Gross $G$-motive (in general $p$-adic coefficient) as the following:
\begin{displaymath}
\mathrm{Factor}_{\mathcal{M}(G)^\mathrm{Gross,Local},\mathrm{Frob}^{-1},\mathcal{E}}(t):=\mathrm{det}(1-t\mathrm{Frob}^{-1}|\mathrm{H}^*(\overline{\mathrm{FF}_{\mathcal{O}_{X,x}}},\mathcal{M}(G)^\mathrm{Gross,Local}_\mathcal{E})^{-1}.
\end{displaymath}
\end{definition}

\begin{remark}
One may want to compare this with the one defined globally as in the previous section. Also one can define things over the disc $\mathrm{Sp}(\widehat{\mathcal{O}}_{X,x})$:
\begin{definition}
We now here assume that the object $\mathcal{E}$ is an augmented commutative $\mathbb{E}_\infty$-algebra over $\mathrm{Sp}(\widehat{\mathcal{O}}_{X,x})$.
\begin{align}
\mathcal{M}(G)^\mathrm{Gross,Local}_\mathcal{E}:= \mathrm{cotangentfibre}(h_*h^*\mathcal{E}).	
\end{align}
where $h: \mathrm{BC}_{\overline{k}}\mathrm{Classifyingstack}_{G(K_{X,x})}=\mathrm{BC}_{\overline{k}}\mathrm{Sp}(\widehat{\mathcal{O}}_{X,x})/{G(K_{X,x})} \rightarrow \mathrm{BC}_{\overline{k}}\mathrm{Sp}(\widehat{\mathcal{O}}_{X,x})$.
\end{definition}
\noindent A little bit weaker result will be the following Euler product formula for the L-functions attached to the rigid Gross $G$-motive (which we will define in our setting which is an $\infty$-$p$-adic arithmetic $\mathcal{D}$-module) which we define in the following:
\begin{definition}
We now here assume that the object $\mathcal{E}$ is an augmented commutative $\mathbb{E}_\infty$-algebra. We define the rigid $L$-factor attached to rigid Gross $G$-motive (in general $p$-adic coefficient) as the following:
\begin{displaymath}
\mathrm{Factor}_{\mathcal{M}(G)^\mathrm{Gross,Local},\mathrm{Frob}^{-1},\mathcal{E}}(t):=\mathrm{det}(1-t\mathrm{Frob}^{-1}|\mathrm{H}^*(\overline{\mathrm{Sp}(\widehat{\mathcal{O}}_{X,x})},\mathcal{M}(G)^\mathrm{Gross,Local}_\mathcal{E})^{-1}.
\end{displaymath}
\end{definition}	
\end{remark}

\newpage

\subsection*{Acknowledgements} 
This work is based on the conversation with Professor Kedlaya who conjectured that there should be a version of Drinfeld's lemma for $F$-isocrystals and suggested taking a look at the corresponding generalization of Abe's work on Langlands correspondence for $F$-isocrystals to general reductive groups. Also we benefit a lot from the corresponding work \cite{Ked9} from different perspectives. We would like to thank Professor Kedlaya for helpful discussion and many suggestions around this problem and related.

\newpage

\bibliographystyle{ams}

\end{document}